\documentclass[12pt]{article}
\pdfoutput=1
\title{The Combinatorial Game Theory of Well-tempered Scoring Games}
\author{Will Johnson}
\usepackage{amsmath, amssymb, amsthm}    	
\usepackage{fullpage} 	
\usepackage{hyperref}
\usepackage{graphicx}	
\usepackage{verbatim}
\usepackage{amscd}
\usepackage{color}
\usepackage{subfig}

\newcommand{\lout}{\operatorname{L}}
\newcommand{\rout}{\operatorname{R}}
\newcommand{\lfout}{\operatorname{Lf}}
\newcommand{\rfout}{\operatorname{Rf}}
\newcommand{\gap}{\operatorname{gap}}
\newcommand{\even}{\operatorname{even}}

\newcommand{\modforreal}{\operatorname{mod}}

\newtheorem{theorem}{Theorem}[section] 
\newtheorem{lemma}[theorem]{Lemma}
\newtheorem{claim}[theorem]{Claim}
\newtheorem{definition}[theorem]{Definition}
\newtheorem{corollary}[theorem]{Corollary}

\newtheorem{fact}[theorem]{Fact}

\begin{document}
\maketitle

\begin{abstract}
We consider the class of ``well-tempered'' integer-valued scoring games, which have the property that the parity of the length of the game is independent of the line of play.  We consider disjunctive sums of these games, and develop a theory for them analogous to the standard theory of disjunctive sums of normal-play partizan games.  We show that the monoid of well-tempered scoring games modulo indistinguishability is cancellative but not a group, and we describe its structure in terms of the group of normal-play partizan games.  We also classify Boolean-valued well-tempered scoring games, showing that there are exactly seventy, up to equivalence.
\end{abstract}

\section{Introduction}
The standard theory of combinatorial game theory focuses on partizan games played under the \emph{normal play rule}.  In these games, two players
take turns moving until one player is unable to.  This player loses.  A substantial amount of theoretical work has gone into these games, as evidenced by
the books \cite{ONAG}, \cite{WinWays}, \cite{GONC} and \cite{MGONC}.  In contrast, the earliest papers on partizan combinatorial games, by Milnor~\cite{Milnor}
and Hanner~\cite{Hanner}, focused on what are now called \emph{scoring games}, which include games like \textsc{Go}
and \textsc{Dots-and-Boxes} to varying degrees.
A scoring game specifies an arbitrary numerical score depending on the ending position of the game, rather than on the identity of the final player.
One player seeks to maximize this score while the other seeks to minimize it.

The papers of Milnor and Hanner focused on scoring games having the property that there is always an incentive to move.
In other words, players would always prefer moving to passing.  With this assumption,
games modulo equivalence form a nice abelian group, one that is closely related to the normal-play partizan games.  This assumption may seem limiting,
but it is natural in games like Go, where there is no compulsion to move.

Later work by Ettinger~\cite{Ettinger1} and \cite{Ettinger2}
considered the general class of scoring games.  In this case, games modulo equivalence form a commutative monoid,
but not a group.  Ettinger was interested in the question of whether this monoid was cancellative, but he seems to have reached no definite conclusion.

The present paper focuses on a limited class of scoring games, which we call \emph{well-tempered} because they are somewhat analogous to the well-tempered
games of Grossman and Siegel~\cite{GrossmanSiegel}.
A well-tempered scoring game is one in which the last player to move is predetermined, and does not depend
on the course of the game.  In an odd-tempered game, the first player to move will also be the last, while in an even-tempered game,
the first player will never be the last.  These are exactly the scoring games which are trivial when played as partizan games
by the normal-play rule. Nevertheless, we show that our class of games,
like the class considered by Milnor, is closely related to the standard theory of partizan normal-play games.

Our original motivation was the game \textsc{To Knot or Not to Knot} of \cite{KnotGames}.
Few other games seem to fit into this framework, so we have invented a few.
We refer the reader to \S \ref{sec:Examples} for examples.

The outline of this paper is as follows.  In Section \ref{sec:Definitions},
we define precisely what we mean by a well-tempered scoring game, how we add them, and what it means for two to be equivalent.
In \S \ref{sec:basic}, we prove basic facts with these notions, showing that a certain class of special games (those having the property
of Milnor and Hanner at even levels) form a well-behaved abelian group.  To a general game $G$, we can associate two special games $G^+$ and
$G^-$, which characterize $G$.  In fact, we show that a general game $G$ has a double identity, acting as $G^+$ when Left is the final player,
and as $G^-$ when Right is the final player.  In Section \ref{sec:distortions}, we consider variants of disjunctive addition.  Specifically, we vary
the manner in which the final scores of the summands are combined into a total score.  We provide additional motivation for our earlier definition
of equivalence, generalize the results of \S\ref{sec:basic}, and characterize the pairs $(G^+,G^-)$ which can occur.  In \S\ref{sec:psi}, we show
how our class of games is closely related to and characterized by the standard theory of normal-play partizan games.  We use this correspondence
to give canonical forms for (special) games in \S\ref{sec:canonical}.  In \S\ref{sec:boolean} we discuss the theory of $\{0,1\}$-valued games,
which is necessary to analyze games like \textsc{To Knot or Not to Knot}.  We show that there are seventy $\{0,1\}$-valued games modulo equivalence,
but infinitely many three-valued games.  We close in \S~\ref{sec:Examples} with some examples of well-tempered scoring games, including small dictionaries for them.

We assume the reader is familiar with the basic notations and conventions of \emph{Winning Ways}~\cite{WinWays} chapters 1-2,
or \emph{On Numbers and Games}~\cite{ONAG} chapters 7-10.
We assume that the reader is capable of making basic inductive arguments of the sort found in chapter 1 of \emph{ONAG}.  We do not assume
that the reader is familiar with thermography, atomic weights, Norton multiplication, or the theory of loopy games, though we mention them occasionally in passing.
We will refer to the normal-play partizan games of the standard theory as ``partizan games.''  We will use $\lhd$ and $\rhd$ for the relations of ``less than
or fuzzy to'' and ``greater than or fuzzy to.''  For partizan games, $=$ and $\equiv$ will denote equivalence and identity, but for scoring games we will
use $\approx$ and $=$, respectively.

\section{Definitions}\label{sec:Definitions}
We are solely concerned with the following class of scoring games:
\begin{definition}\label{def:wtsg}
An \emph{even-tempered game} is an integer or a pair $\{G^L|G^R\}$ where $\{G^L\}$ and $\{G^R\}$ are finite nonempty sets
of odd-tempered games.\footnote{We are using the standard notational abbreviations of combinatorial game theory, where
$G^L$ and $G^R$ are variables ranging over the left and right options of a game $G$.  We will use this kind of abbreviation later
in Definitions~\ref{def:disj} and \ref{def:neg} to save space.}  An \emph{odd-tempered game} is a pair $\{G^L|G^R\}$ where $\{G^L\}$ and $\{G^R\}$ are finite nonempty
sets of even-tempered games.  A \emph{well-tempered game} is an even-tempered game or an odd-tempered game.
\end{definition}
We call the elements of $\{G^L\}$ and $\{G^R\}$ the \emph{left and right options} of $G$.  If $G$ is an integer, we define its options
to be the elements of $\emptyset$.  If $G$ is a well-tempered game, we let $\pi(G) = 0$ if $G$ is even-tempered and $\pi(G) = 1$ if $G$ odd-tempered.


In what follows, all ``games'' will be well-tempered scoring games, unless specified otherwise.  
We view $\{G^L|G^R\}$ as a game between two players, Left and Right.  Left can move to any left option $G^L$ and Right can move to any right
option $G^R$.  Once the game reaches an integer, this value becomes the final score.  We assume the Left is maximizing the score while Right is minimizing it.

For example, the game $\{0|1\}$ is an odd-tempered game that will last exactly one turn.  If Left goes first, she will get 0 points, but if Right goes first,
Left will get 1 point.  On the other hand, in $\{2|-2\}$, whichever player goes first will get two points.  The odd-tempered game $\{2,3|\{1|1||2|2\}\}
 = \{2,3|\{\{1|1\}|\{2|2\}\}\}$ will last one
turn if Left goes first, or three turns if Right goes first.  Left can choose to either receive 2 or 3 points.  If Right goes first, the final score will
necessarily be 1, because play will proceed
\[ \{2,3|\{1|1||2|2\}\} \to \{1|1||2|2\} \to \{1|1\} \to 1\]
since the players move in alternation.

With these conventions, it is natural to define
\begin{definition}
The \emph{left outcome} and \emph{right outcome} of $G$, denoted $\lout(G)$ and $\rout(G)$, respectively, are recursively defined as follows:
\[
\lout(G) = \begin{cases}
G & \text{ if $G$ is an integer} \\
\max_{G^L} \rout(G^L) \text{ otherwise }
\end{cases}\]
\[
\rout(G) = \begin{cases}
G & \text{ if $G$ is an integer} \\
\min_{G^R} \lout(G^R) \text{ otherwise }
\end{cases}\]
Here $\max_{G^L} \rout(G^L)$ denotes the maximum value of $\rout(G^L)$ as $G^L$ ranges over the left options of $G$, and similarly
$\min_{G^R} \lout(G^R)$ is the minimum as $G^R$ ranges over the right options of $G$.

We define the \emph{outcome} of $G$ to be the pair $(\lout(G),\rout(G))$.
\end{definition}

The left and right outcomes of $G$ are the scores that result under perfect play when Left or Right moves \emph{first}.
In a well-tempered game, it also makes sense to discuss the player who moves \emph{last}, so we can make the following definition:
\begin{definition}
The \emph{left final outcome} of $G$, denoted $\lfout(G)$, is $\lout(G)$ when $G$ is odd and $\rout(G)$ when $G$ is even.  Similarly,
the \emph{right final outcome} of $G$, denoted $\rfout(G)$, is $\rout(G)$ when $G$ is odd and $\lout(G)$ when $G$ is even.
\end{definition}
In other words, the left or right final outcome of $G$ is the outcome when Left or Right is the \emph{final} player, rather than the \emph{first} player
for ordinary outcomes.  Note that $(\lout(G),\rout(G),\pi(G))$ carries the same information as $(\lfout(G),\rfout(G),\pi(G))$.
\begin{definition}
Let $G$ be a well-tempered scoring game.  A \emph{subgame of $G$} is $G$ or a subgame of an option of $G$.
If $S$ is a set of numbers, we say that $G$ \emph{takes values in $S$} or is \emph{$S$-valued} if every subgame of $G$
that is an integer is an element of $S$.  We denote the class of $S$-valued well-tempered scoring games $\mathcal{W}_S$.
\end{definition}
Note that $\mathcal{W}_\mathbb{Z}$ is the class of all well-tempered scoring games.  Also, every game has finitely many subgames,
and in particular is in $\mathcal{W}_S$ for some finite $S$.

The basic operation we would like to apply to our games is the disjunctive sum, in which we play two games in parallel, and add the final scores.
We will first define a slightly more general operation.
\begin{definition}\label{def:disj}
Let $S_1, S_2 \subseteq f$, and let $f : S_1 \times S_2 \to \mathbb{Z}$ be a function.  Then let \emph{the extension of $f$ to games} be the
function $\tilde{f} : \mathcal{W}_{S_1} \times \mathcal{W}_{S_2} \to \mathcal{W}_{\mathbb{Z}}$ defined recursively by
\[ \tilde{f}(G_1,G_2) =
\begin{cases}
f(G_1,G_2) & \text{ if $G_1$ and $G_2$ are integers} \\
\{\tilde{f}(G_1^L,G_2),\tilde{f}(G_1,G_2^L) | \tilde{f}(G_1^R,G_2), \tilde{f}(G_1,G_2^R)\} & \text{ otherwise}
\end{cases}\]
\end{definition}

The game $\tilde{f}(G,H)$ is obtained by playing $G$ and $H$ in parallel, and combining the final scores using $f$.  At each turn, the current player
chooses one of $G$ or $H$ to move in.  So in some sense, $\tilde{f}$ is like the disjunctive sum of normal-play partizan games.
In the case where $S_1, S_2 = \mathbb{Z}$ and $f(x,y) = x+y$, we denote $\tilde{f}(G,H)$ by $G + H$.
We call this the \emph{disjunctive sum} of the scoring games $G$ and $H$.

Similarly, if $f : S_1 \times S_2 \times \cdots S_n \to \mathbb{Z}$ is a function, then there is an extension to games
\[ \tilde{f} : \mathcal{W}_{S_1} \times \cdots \times \mathcal{W}_{S_n} \to \mathcal{W}_{\mathbb{Z}}\]
Note that $\tilde{f}(G_1,\ldots,G_n)$ is indeed a well-tempered game, with parity given as follows:
\[ \pi(\tilde{f}(G_1,\ldots,G_n)) \equiv \pi(G_1) + \cdots + \pi(G_n) \quad (\mod 2).\]
So in particular, $\pi(G + H) \equiv \pi(G) + \pi(H)$ (mod 2).

We mainly care about $f$ which are order-preserving in the following sense:
\begin{definition}
Let $f : S_1 \times \cdots \times S_n \to \mathbb{Z}$ be a function.  Then we say that $f$ is \emph{order-preserving} if
\[ f(x_1',\ldots,x_n') \ge f(x_1,\ldots,x_n)\]
whenever $x_i' \ge x_i$ for all $i$.
If $f$ is order-preserving, we call the extension to games $\tilde{f}$ a \emph{generalized disjunctive operation}.
\end{definition}

Another basic operation on games is negation, which reverses the scores and interchanges the roles of the players:
\begin{definition}\label{def:neg}
If $G$ is an integer, let $-G$ be the negative of $G$ in the usual sense.  Otherwise, define $-G$ by
\[ -G = \{-G^R|-G^L\}.\]
\end{definition}
We let $G - H$ denote $G + (-H)$.

The main obstacle to directly applying classic combinatorial game theory to scoring games is that a game like $\{-1|1\}$ would be
equivalent to 0 in the standard theory, but this fails in the theory of scoring games.  The problem centers around games that are ``cold'' in the sense
that there is no incentive to move.  We can measure this coldness as follows:
\begin{definition}
Let $G$ be a well-tempered game.  For $i = 0,1$, let $\gap_i(G)$ be the supremum
of $\rout(H) - \lout(H)$ as $H$ ranges over the subgames of $G$ with $\pi(H) = i$.
We call $\gap_0(G)$ the \emph{even gap} of $G$ and $\gap_1(G)$ the \emph{odd gap} of $G$.
\end{definition}
Note that if $G$ is an integer, $\gap_0(G) = 0$ and $\gap_1(G) = -\infty$.  Since every game has at least one subgame that is an integer,
$\gap_0(G) \ge 0$ for every $G$.  Moreover, every game that isn't an integer satisfies $\gap_1(G) > -\infty$. Also note that if
$H$ is a subgame of $G$, then $\gap_i(H) \le \gap_i(G)$.

Another useful operation on games is heating.
\begin{definition}
If $G$ is a game and $t$ is an integer, then $\int^t G$ (\emph{$G$ heated by $t$}) is the game defined inductively by
\[ \int^t G = \begin{cases}
G & \text{ if } G \in \mathbb{Z} \\
\left\{t + \int^t G^L | -t + \int^t G^R \right\} & \text{ otherwise}
\end{cases}\]
\end{definition}
The effect of this operation is to give every player a bonus of $t$ points for each move she makes.  We allow $t$ to be negative,
in which case we refer to this operation as \emph{cooling}.  Note that our cooling operation is much simpler than the cooling
operation of the standard theory.

One of the central ideas behind combinatorial game theory is to consider games modulo equivalence.  For us, the appropriate definition
of equivalence is the following:
\begin{definition}
Let $G$, $H$ be games.  We say that $G$ and $H$ are \emph{equivalent}, denoted $G \approx H$, if $\lout(G + X) = \lout(H + X)$
and $\rout(G + X) = \rout(H + X)$ for every game $X$.  Similarly, we say that $G \gtrsim H$ if $\lout(G + X) \ge \lout(H + X)$
and $\rout(G + X) \ge \rout(H + X)$ for every game $X$.  We define $G \lesssim H \iff H \gtrsim G$.
\end{definition}
The relation $\approx$ is clearly an equivalence relation, and $\gtrsim$ is a preorder which induces a partial order on $\mathcal{W}_\mathbb{Z}/\approx$.



To motivate Definition~\ref{def:half-eq}, we need a basic fact:
\begin{theorem}\label{parity-incomparable}
Let $G, H$ be games.  If $G \lesssim H$, then $\pi(G) = \pi(H)$.
\end{theorem}
\begin{proof}
Let $K$ be any game, and consider the game $K + \{-N|N\}$ for $N$ a large integer.  Clearly, neither player wants
to move in $\{-N|N\}$, and if $N \gg 0$, both players will make it their top priority to \emph{not} move in $\{-N|N\}$.  Consequently,
the final player will be forced to move in $\{-N|N\}$ and will pay a heavy penalty. So for $K$ arbitrary,
\[ \lim_{N \to +\infty} \lfout(K + \{-N|N\}) = -\infty\]
\[ \lim_{N \to +\infty} \rfout(K + \{-N|N\}) = +\infty.\]
Now suppose that $G \lesssim H$ and $\pi(G) \ne \pi(H)$.  If $G$ is even and $H$ is odd, then
\[ \rfout(G + \{-N|N\}) = \rout(G + \{-N|N\}) \le \rout(H + \{-N|N\}) = \lfout(G + \{-N|N\})\]
for all $N$, contradicting the fact that the left hand side goes to $+\infty$ while the right hand side goes to $-\infty$.
If $G$ is odd and $H$ is even, then
\[ \rfout(G + \{-N|N\}) = \lout(G + \{-N|N\}) \le \lout(H + \{-N|N\}) = \lfout(H + \{-N|N\})\]
which again contradicts the limiting behavior as $N \to +\infty$.
\end{proof}

\begin{definition}\label{def:half-eq}
If $G, H$ are games, we let $G \approx_+ H$ if $\pi(G) = \pi(H)$ and $\lfout(G + X) = \lfout(H + X)$ for every $X$.  Similarly,
we let $G \gtrsim_+ H$ if $\pi(G) = \pi(H)$ and $\lfout(G + X) \ge \lfout(H + X)$ for every $H$.  Define $G \lesssim_+ H$ if $H \gtrsim_+ G$.

Define $\approx_-$, $\lesssim_-$, and $\gtrsim_-$ analogously, using $\rfout(G + X)$ and $\rfout(H +X)$ instead of $\lfout(G + X)$ and $\lfout(H + X)$.
We call $\approx_+$ and $\approx_-$ \emph{left-equivalence} and \emph{right-equivalence}, respectively.
\end{definition}
Note that if $G \lesssim_+ H$ and $G \lesssim_- H$, then $G \lesssim H$.  The converse follows by Theorem~\ref{parity-incomparable},
and analogous statements hold for $\approx$ and $\approx_\pm$.

\section{Games modulo equivalence}\label{sec:basic}
\subsection{Basic Facts}\label{subsec:basic}
\begin{lemma}
Let $G, H, K \in \mathcal{W}_\mathbb{Z}$, and $s, t \in \mathbb{Z}$.  Then we have the following identities
(not equivalences):
\begin{align}
G + H &= H + G \\
(G + H) + K &= G + (H + K) \label{add-assoc}\\
-(G + H) &= (-G) + (-H) \label{add-neg}\\
-(-G) &= G \\
G + 0 &= G \\
\int^0 G &= G \\
\int^t 0 &= 0 \\
\int^t(G + H) &= \int^ G + \int^t H \label{heat-hom}\\
\int^t (-G) &= -\int^t G \label{heat-neg}\\
\int^t \int^s G &= \int^{t + s} G \label{double-heat}
\end{align}
\end{lemma}
\begin{proof}
All of these identities have easy inductive proofs akin to those found in Chapter 1 of \cite{ONAG}.  To give an example, we prove (\ref{double-heat}).
In the base case where $G$ is an integer, (\ref{double-heat}) is trivial.  Otherwise,
\begin{align*} \int^t \int^s G &= \left\{ t + \int^t\left(\left(\int^s G\right)^L\right) | -t + \int^t\left(\left(\int^s G\right)^R\right)\right\} \\
&= \left\{t + \int^t\left(s + \int^s G^L\right)|-t + \int^t\left(-s + \int^s G^R\right)\right\} \\
&= \left\{t + \left(\int^t s + \int^t \int^s G^L\right) | -t + \left(-\int^t s + \int^t \int^s G^R\right)\right\} \\
&= \left\{(t + s) + \int^t \int^s G^L | -(t + s) + \int^t \int^s G^R\right\} \\
&= \left\{(t + s) + \int^{t+s} G^L | -(t + s) + \int^{t + s} G^R\right\} = \int^{t + s} G\end{align*}
where the penultimate equality follows by induction, and we have also used (\ref{add-assoc}), (\ref{add-neg}), (\ref{heat-hom}), and (\ref{heat-neg}).
\end{proof}

\begin{lemma}
If $G_1, G_2,\ldots$ are games, $t \in \mathbb{Z}$, and $\tilde{f}$ is a generalized disjunctive operation, then
\begin{align}
\pi(G_1 + G_2) &\equiv \pi(G_1) + \pi(H_2) \quad (\modforreal 2) \\
\pi(-G_1) &= \pi(G_1) \\
\pi\left(\int^t G_1\right) &= \pi(G_1) \\
\pi(\tilde{f}(G_1,G_2,\ldots,G_n)) &\equiv \sum_{i = 1}^n \pi(G_i) \quad (\modforreal 2)
\end{align}
\end{lemma}
Again the proofs are easy inductions, so we omit them.

\begin{lemma}\label{early-lemma}
Let $G$ be a game, $n, t \in \mathbb{Z}$, $i \in \{0,1\}$, and $f, g$ be order-preserving.  Then
\begin{align}
\lout(-G) &= -\rout(G) \label{lkjfirst} \\
\rout(-G) &= -\lout(G) \\
\lout(G + n) &= \lout(G) + n \label{lkjmiddle}\\
\lout(\tilde{f}(G)) &= f(\lout(G)) \\
\lout(\tilde{g}(n,G)) &= g(n,\lout(G)) \label{lkjlast}\\
\lout\left(\int^t G\right) &= \lout(G) + t \cdot \pi(G) \label{heat-out-one}\\
\rout\left(\int^t G\right) &= \rout(G) - t \cdot \pi(G) \label{heat-out-two}\\
\gap_i\left(\int^t G\right) &= \gap_i(G) - 2t \cdot i \label{heat-gap}\\
\gap_i(-G) &= \gap_i(G) \label{gap-neg}\\
\end{align}
Also, (\ref{lkjmiddle}-\ref{lkjlast}) hold with $\rout$ instead of $\lout$, and (\ref{lkjfirst}-\ref{lkjlast}) hold
with $\lfout$ and $\rfout$ in place of $\lout$ and $\rout$.
\end{lemma}
Again, all the proofs are easy inductions, so we omit them.

\begin{lemma}\label{obvious-compatibles}
Let $\Box$ be any of the following operations $\approx, \gtrsim, \lesssim, \approx_\pm, \gtrsim_\pm, \lesssim_\pm$,
and let $\tilde{\Box}$ denote $\approx, \gtrsim, \lesssim, \approx_\mp, \gtrsim_\mp, \lesssim_\mp$, respectively.  Let $G, H, K$
be games and suppose that $G \Box H$.  Then
\begin{align}
G + K &\Box H + K \label{add-compat}\\
-H &\tilde{\Box} -G \label{neg-compat}\\
\int^t G &\Box \int^t H  \label{heat-compat}
\end{align}
Also, if $G_i \Box H_i$ for every $i$, then
\begin{equation} \{G_1,\ldots,G_n|G_{n+1},\ldots,G_m\} \Box \{H_1,\ldots,H_n|H_{n+1},\ldots,H_m\}\label{build-compat}\end{equation}
\end{lemma}
\begin{proof}
For example, suppose that $\Box$ is $\gtrsim_+$.  If $G \gtrsim_+ H$, then for any game $X$,
\[ \lfout((G + K) + X) = \lfout(G + (K + X)) \ge \lfout(H + (K + X)) = \lfout((H + K) + X),\]
so (\ref{add-compat}) holds.  Similarly,
\[ \rfout(-H + X) = \rfout(-(H + (-X))) = -\lfout(H + (-X)) \ge -\lfout(G + (-X)) = \rfout(-G + X),\]
so $-H \gtrsim_- -G$ and (\ref{neg-compat}) holds.  For (\ref{heat-compat}), note that
\begin{align*} \lfout\left(\int^t G + X\right) - \lfout\left(\int^t H + X\right) &= \lfout\left(\int^t\left(G + \int^{-t} X\right)\right)
- \lfout\left(\int^t\left(H + \int^{-t}X\right)\right) \\
&= \lfout\left(G + \int^{-t} X\right) - \lfout\left(H + \int^{-t} X\right) \ge 0\end{align*}
using either (\ref{heat-out-one}) or (\ref{heat-out-two}) and the fact that $G$ and $H$ have the same parity.  Thus (\ref{heat-compat}) holds.

For the final claim, one can show by an easy induction on $X$ that
\[ \lfout(\{G_1,\ldots,G_n|G_{n+1},\ldots,G_m\} + X) \ge \lfout(\{H_1,\ldots,H_n|H_{n+1},\ldots,H_m\} + X).\]
We leave this as an exercise to the reader.  Having shown the lemma for $\gtrsim_+$, the case of $\lesssim_+$ follows immediately,
$\gtrsim_-$ and $\lesssim_-$ follow by symmetry, and the remaining five relations follow because they are logical conjunctions of these four.
\end{proof}
An immediate consequence of this is that $\mathcal{W}_{\mathbb{Z}}/\approx$, $\mathcal{W}_{\mathbb{Z}}/\approx_+$,
and $\mathcal{W}_{\mathbb{Z}}/\approx_-$ are partially ordered monoids with addition induced by $+$ and partial orders induced
by $\gtrsim$, $\gtrsim_+$, or $\gtrsim_-$ respectively.  Also, the game-building operation $\{\cdots|\cdots\}$ is well-defined on
each of these monoids, as are heating, cooling, and negation.  We will see in Theorem~\ref{generalized-compatibility} that generalized disjunctive operations
are also compatible with the nine relations.

\subsection{Playing two games at once}
\begin{lemma}\label{mirrors}
If $G$ is a game, then
\[ \rout(G - G) \ge 0\]
\[ \lout(G - G) \le 0\]
\end{lemma}
\begin{proof}
If Left goes second, she can respond to any move in $G$ with the corresponding move in $-G$, following this strategy till the game ends.
This ensures her a score of at least 0, though she could possibly do better.  So $\rout(G - G) \ge 0$.  A similar argument shows
$\lout(G - G) \le 0$.
\end{proof}

In classic combinatorial game theory, a simple argument shows that $G + H \ge 0$ if $G \ge 0$ and $H \ge 0$.  The analogous result for scoring games
is the following confusing lemma, which encompasses about a dozen inequalities all proven in analogous ways:
\begin{lemma}\label{combo-strat}
Let $G, H$ be well-tempered games.  Then
\begin{align} \rout(G + H) &\ge \rout(G) + \rout(H) - \max(\pi(H) \cdot \gap_{\pi(G)}(G),\pi(G) \cdot \gap_{\pi(H)}(H))\label{first} \\
\lout(G + H) &\ge \lout(G) + \rout(H) - \max(\pi(H) \cdot \gap_{1-\pi(G)}(G),(1-\pi(G)) \cdot \gap_{\pi(H)}(H))\label{second}\end{align}
where $i \cdot n$ equals $n$ if $i = 1$ and $0$ if $i = 0$, even if $n = -\infty$.
\end{lemma}
We list the four cases of these equations for future reference:
If $G, H$ are even, then
\begin{align} \rout(G + H) &\ge \rout(G) + \rout(H) \label{eer} \\
 \lout(G + H) &\ge \lout(G) + \rout(H) - \gap_0(H).\label{eel}\end{align}
If $G, H$ are odd, then
\begin{align} \rout(G + H) &\ge \rout(G) + \rout(H) - \max(\gap_1(G),\gap_1(H)) \label{oor}\\
 \lout(G + H) &\ge \lout(G) + \rout(H) - \gap_0(G)\label{ool}\end{align}
If $G$ is odd and $H$ is even, then
\begin{align} \rout(G + H) &\ge  \rout(G) + \rout(H) - \gap_0(H)\label{oer}\\
 \lout(G + H) &\ge  \lout(G) + \rout(H) \label{oel}\end{align}
Finally, if $G$ is even and $H$ is odd, then
\begin{align} \rout(G + H) &\ge  \rout(G) + \rout(H) - \gap_0(G)\label{eor}\\
 \lout(G + H) &\ge \lout(G) + \rout(H) - \max(\gap_1(G),\gap_1(H))\label{eol}\end{align}
Here we have used the fact that $\max(0,\gap_0(H)) = \gap_0(H)$ because $\gap_0(H) \ge 0$.
\begin{proof}
We prove (\ref{first}-\ref{second}) together by induction on the complexity of $G$ and $H$.  First suppose that $G$ is an integer.  Then
\begin{align*} \max(\pi(H)\gap_{\pi(G)}(G),\pi(G)\gap_{\pi(H)}(H)) &= \max(\pi(H)\gap_0(G),0\gap_{\pi(H)}(H)) \\ &= \max(\pi(H)0,0) = 0.\end{align*}
So (\ref{first}) says
\[ \rout(G + H) \ge \rout(G) + \rout(H),\]
which holds because $\rout(n+H) = n + \rout(H) = \rout(n) + \rout(H)$ for any integer $n$, by a variant of (\ref{lkjmiddle}).  So (\ref{first}) holds in this case.
A similar argument shows that (\ref{first}) holds when $H$ is an integer.  Suppose neither $G$ nor $H$ is an integer.  Then
\[ \rout(G +H) = \lout((G+H)^R)\]
for some right option $(G+H)^R$ of $G+H$.  The option $(G+H)^R$ is either of the form $G^R + H$ or $G + H^R$.  By symmetry, we can assume
the former.  Then
\[ \rout(G + H) = \lout(G^R + H) \ge \lout(G^R) + \rout(H) - \max(\pi(H)\gap_{1-\pi(G^R)}(G^R),(1-\pi(G^R))\gap_{\pi(H)}(H))\]
by induction.  But $\pi(G^R) = 1 - \pi(G)$, and $\gap_i(G^R) \le \gap_i(G)$, so
\begin{align*} \rout(G + H) &\ge \lout(G^R) + \rout(H) - \max(\pi(H)\gap_{\pi(G)}(G^R),\pi(G)\gap_{\pi(H)}(H)) \\ &\ge \lout(G^R) + \rout(H) - \max(\pi(H)\gap_{\pi(G)}(G),
\pi(G)\gap_{\pi(H)}(H)).\end{align*}
But since $\rout(G) = \min_{G^R} \lout(G^R)$, we must have $\lout(G^R) \ge \rout(G)$.  So
\[ \rout(G + H) \ge \rout(G) + \rout(H) - max(\pi(H)\gap_{\pi(G)}(G),
\pi(G)\gap_{\pi(H)}(H))\]
as desired.  Thus (\ref{first}) holds.

Next, we show that (\ref{second}) holds.  First suppose that $G$ is an integer.  Then
\[ \max(\pi(H)\gap_{1 - \pi(G)}(G),(1-\pi(G))\gap_{\pi(H)}(H)) = \max(\pi(H)\gap_1(G),\gap_{\pi(H)}(H)).\]
If $H$ is odd-tempered, this equals $\max(-\infty,\gap_1(H)) = \gap_1(H)$.  If $H$ is even-tempered, this equals $\max(0,\gap_0(H)) = \gap_0(H)$.  Either way,
it equals $\gap_{\pi(H)}(H)$. So we want to show that
\[ \lout(G + H) \ge \lout(G) + \rout(H) - \gap_{\pi(H)}(H).\]
But since $G$ is an integer, $\lout(G + H) = G + \lout(H) = \lout(G) + \lout(H)$.  So it remains to show that
\[ \lout(H) - \rout(H) \ge -\gap_{\pi(H)}(H)\]
which is true by definition of $\gap_i(H)$.

Finally, suppose that $G$ is not an integer.  Then $\lout(G) = \rout(G^L)$ for some left option $G^L$ of $G$.  As $G^L + H$ is a left option
of $G + H$,
\[ \lout(G + H) \ge \rout(G^L + H).\]
By induction, it follows that
\begin{align*} \lout(G + H) &\ge \rout(G^L) + \rout(H) - \max(\pi(H)\gap_{\pi(G^L)}(G^L),\pi(G^L)\gap_{\pi(H)}(H))
 \\ &= \lout(G) + \rout(H) - \max(\pi(H)\gap_{\pi(G^L)}(G^L),\pi(G^L)\gap_{\pi(H)}(H))\end{align*}
Now $\pi(G^L) = 1 - \pi(G)$ and $\gap_i(G^L) \le \gap_i(G)$, so
\begin{align*} \lout(G + H) &\ge \lout(G) + \rout(H) - \max(\pi(H)\gap_{1 - \pi(G)}(G^L),(1-\pi(G))\gap_{\pi(H)}(H)) \\ &\ge
\lout(G) + \rout(H) - \max(\pi(H)\gap_{1 - \pi(G)}(G),(1-\pi(G))\gap_{\pi(H)}(H))\end{align*}
as desired.
\end{proof}

By symmetry, we also get the following:
\begin{align*} \lout(G + H) &\le \lout(G) + \lout(H) + \max(\pi(H)\gap_{\pi(G)}(G),\pi(G)\gap_{\pi(H)}(H))
\\ \rout(G + H) &\le \rout(G) + \lout(H) + \max(\pi(H)\gap_{1-\pi(G)}(G),(1-\pi(G))\gap_{\pi(H)}(H))\end{align*}
There are corresponding duals of (\ref{eer}-\ref{eol}).  We will refer to these dual inequalities using a $\star$.  For example, (\ref{eor}$\star$) says
that if $G$ is even and $H$ is odd, then
\[ \lout(G + H) \le \lout(G) + \lout(H) + \gap_0(G) \]

\subsection{Games with small gaps}
Using the inequalities of the previous section, we can examine the interaction of the gap functions with addition:
\begin{theorem}\label{gap-sums}
If $G, H$ are games, then $\gap_0(G + H) \le \gap_0(G) + \gap_0(H)$ and $\gap_1(G + H) \le \max(\gap_1(G), \gap_1(H))$.
\end{theorem}
\begin{proof}
We prove each statement separately by induction.  First consider $\gap_0(G + H)$.  If $g = \gap_0(G)$ and $h = \gap_0(H)$,
then every option $G'$ of $G$ has $\gap_0(G') \le g$ and every option $H'$ of $H$ has $\gap_0(H') \le h$.  So by induction,
every option of $G + H$ has $\gap_0 \le g + h$, and it remains to show that if $G + H$ is even-tempered, then $\lout(G + H) - \rout(G + H) \ge -(g + h)$.

If $G$ and $H$ are both even, then by (\ref{eel})
\[ \lout(G + H) \ge \lout(G) + \rout(H) - h\]
and by (\ref{eel}$\star$)
\[ \rout(H + G) \le \rout(H) + \lout(G) + g.\]
Thus
\[ \lout(G + H) - \rout(G + H) \ge -(h + g)\]
as desired.  Similarly, if $G$ and $H$ are both odd, then by (\ref{ool}) and (\ref{ool}$\star$),
\[ \lout(G + H) \ge \lout(G) + \rout(H) - g\]
\[ \rout(H + G) \le \rout(H) + \lout(G) + h\]
so
\[ \lout(G + H) - \rout(G + H) \ge -(h+g)\]
as desired.  If $G$ and $H$ have different parity, then $G + H$ is not even and there is nothing to show.

Next, suppose that $\gap_1(G), \gap_1(H) \le m$.  We want to show that $\gap_1(G + H) \le m$.  Every option of $G$ and $H$ also has
$\gap_1 \le m$, so by induction, every option of $G + H$ has $\gap_1 \le m$.  So it remains to check
that if $G + H$ is odd, then $\lout(G + H) - \rout(G + H) \ge -m$.  Without loss of generality, $G$ is odd and $H$ is even.  Then
by (\ref{eol}) and (\ref{oel}$\star$),
\[ \lout(H + G) \ge \lout(H) + \rout(G) - \max(\gap_1(H),\gap_1(G)) \ge \lout(H) + \rout(G) - m,\]
\[ \rout(G+H) \le \rout(G) + \lout(H)\]
Then
\[ \lout(G + H) - \rout(G + H) \ge -m\]
as desired.
\end{proof}

Let $\mathcal{I}$ be the class of games with $\gap_0(G) = 0$, and $\mathcal{J}$ be the class of games with
$\gap_0(G) = 0$ and $\gap_1(G) \le 2$.  By the previous theorem, $\mathcal{I}$ and $\mathcal{J}$ are closed under addition.
They are also closed under negation, by (\ref{gap-neg}).  Note that if $G \in \mathcal{I}$,
then $\int^t G \in \mathcal{J}$ for $t \gg 0$, by (\ref{heat-gap}).

The class $\mathcal{J}$ will be important in \S \ref{sec:psi}, but for now, we focus on the nice properties of $\mathcal{I}$.
\begin{lemma}\label{i-games-nice}
If $G \in \mathcal{I}$ is an even game, then $G \gtrsim 0$ if and only if $\rout(G) \ge 0$.
\end{lemma}
\begin{proof}
If $G \gtrsim 0$ then $\rout(G) \ge 0$ trivially.  Conversely, suppose $\rout(G) \ge 0$.
We need to show that $\lout(G + X) \ge \lout(X)$ and $\rout(G + X) \ge \rout(X)$ for every $X$.  If $X$ is even, these follow
by (\ref{eel}) applied to $X$ and $G$, and (\ref{eer}) applied to $G$ and $X$, respectively.  If $X$ is odd, these follow from (\ref{oel}) applied to $X$ and $G$,
and (\ref{oer}) applied to $G$ and $X$, respectively.
\end{proof}

\begin{theorem}\label{i-games-nice-2}
If $G, H \in \mathcal{I}$, then $G - G \approx 0$ and the following are equivalent:
\begin{description}
\item[(a)] $G \gtrsim H$
\item[(b)] $G \gtrsim_- H$
\item[(c)] $\rout(G - H) \ge 0$ and $\pi(G) = \pi(H)$
\item[(d)] $G \gtrsim_+ H$
\item[(e)] $\lout(H - G) \le 0$ and $\pi(G) = \pi(H)$.
\end{description}
\end{theorem}
\begin{proof}
Since $G$ and $-G$ are in $\mathcal{I}$ and $\pi(G) = \pi(-G)$, $G - G$ is also in $\mathcal{I}$ and is even-tempered.  By Lemma~\ref{mirrors},
$\rout(G - G) \ge 0$.  So by Lemma~\ref{i-games-nice}, $G - G \gtrsim 0$.  By symmetry, $G - G \lesssim 0$, so $G - G \approx 0$.

Next we show that the five statements are equivalent.  $(a) \implies (b)$ is easy, using Theorem~\ref{parity-incomparable}.  For $(b) \implies (c)$,
note that if $G \gtrsim_- H$ then
\[ \rout(G - H) \gtrsim \rout(H - H) = 0,\]
since $H - H \approx 0$.  Also, $\pi(G) = \pi(H)$ by definition of $\gtrsim_-$.  For $(c) \implies (a)$, note that by Lemma~\ref{i-games-nice},
if $\rout(G - H) \ge 0$ and $\pi(G - H) = 0$, then $G - H \gtrsim 0$.  But then by Lemma~\ref{obvious-compatibles} and the fact that $H - H \approx 0$,
\[ G = G + 0 \approx G + (H - H) = H + (G - H) \gtrsim H + 0 = H,\] so $(a)$ holds.

We have shown that $(a)$, $(b)$, and $(c)$ are equivalent.  The equivalence of $(d)$ and $(e)$ follows by symmetry.
\end{proof}
In particular, we see that every element of $\mathcal{I}$ is an invertible element of the monoid $\mathcal{W}_{\mathbb{Z}}/\approx$.  Also,
the relations $\approx$ and $\approx_\pm$ are all equivalent for games in $\mathcal{I}$, as are the relations $\gtrsim$ and $\gtrsim_\pm$.

\subsection{Upsides and Downsides}
Next, we turn to a characterization of the structure of $\mathcal{W}_{\mathbb{Z}}/\approx$ and $\mathcal{W}_{\mathbb{Z}}/\approx_\pm$ in terms
of the group $\mathcal{I}/\approx$.  The key step is the following lemma:
\begin{lemma}\label{gap-rule}
If $G$ is even, and every option of $G$ is in $\mathcal{I}$, and $\lout(G) \le \rout(G)$, then
\[ G \approx_+ \rout(G)\]
\[ G \approx_- \lout(G)\]
\end{lemma}
\begin{proof}
Note that
\begin{equation}
\gap_0(G) = \rout(G) - \lout(G) \label{gap-of-G}
\end{equation}
because the only subgame $G'$ of $G$ for which $\rout(G') - \lout(G')$ could be positive is $G$ itself, as all options of $G$ are in $\mathcal{I}$.

By symmetry, it suffices to prove $G \approx_+ \rout(G)$.
We need to show that if $X$ is any game, then
\[ \lfout(G + X) \stackrel{?}{=} \lfout(\rout(G) + X) = \rout(G) + \lfout(X).\]
If $X$ is even, then by (\ref{eer})
\begin{equation} \lfout(G + X) = \rout(G + X) \ge \rout(G) + \rout(X) = \rout(G) + \lfout(X).\label{eq:temp}\end{equation}
But by (\ref{eel}$\star$) and (\ref{gap-of-G}),
\[ \rout(X + G) \le \rout(X) + \lout(G) + \gap_0(G) = \rout(X) + \lout(G) + \rout(G) - \lout(G) = \rout(X) + \rout(G).\]
So equality holds in (\ref{eq:temp}), and $\lfout(G + X) = \rout(G) + \lfout(X)$ as desired.

Otherwise, $X$ is odd.  Then by (\ref{oel})
\begin{equation} \lfout(G + X) = \lout(X + G) \ge \lout(X) + \rout(G) = \rout(G) + \lfout(X).\label{eq:temp2}\end{equation}
But by (\ref{oer}$\star$) and (\ref{gap-of-G}),
\[ \lout(X + G) \le \lout(X) + \lout(G) + \gap_0(G) = \lout(X) + \lout(G) + (\rout(G) - \lout(G)) = \lout(X) + \rout(G).\]
So equality holds in (\ref{eq:temp2}), and $\lfout(G + X) = \rout(G) + \lfout(X)$ as desired.
\end{proof}
This yields the following key result:
\begin{theorem}\label{sidling}
If $G$ is any game, then there exist games $G^+, G^- \in \mathcal{I}$ with $G^+ \approx_+ G \approx_- G^-$.
If $S \subseteq \mathbb{Z}$ and $G$ is $S$-valued, then $G^+$ and $G^-$ can be chosen to be $S$-valued.
Additionally, $G^+ \gtrsim G^-$.
\end{theorem}
\begin{proof}
By symmetry, we only show the existence of $G^+$.  We proceed by induction on $G^+$.  If $G$ is an integer,
then we can take $G^+ = G$.  Otherwise, by induction every option of $G$ is left-equivalent to an $S$-valued game in $\mathcal{I}$.
Let $H$ be the game obtained from $G$ by replacing every option of $G$ with such a game.  Then by Lemma~\ref{obvious-compatibles}
applied to $\approx_+$, $G$ is left-equivalent to $H$.  Since $H$ is an $S$-valued game, we are done unless $H$ fails to be in $\mathcal{I}$.
But every option of $H$ is in $\mathcal{I}$, so this can only occur if $H$ is even and $\lout(H) < \rout(H)$.  In this case, we have
\[ G \approx_+ H \approx_+ \rout(H)\]
by the previous lemma.  Since $\rout(H) \in S$, we can take the integer $\rout(H)$ to be $G^+$.

For the final claim, note that $G \approx_- G^-$.  Therefore, by Lemma~\ref{obvious-compatibles}
$-G \approx_+ -(G^-)$ and so $G - G \approx_+ (G^+ - G^-)$.  But by Lemma~\ref{mirrors},
\[ \rout(G^+ - G^-) = \lfout(G^+ - G^-) = \lfout(G - G) = \rout(G - G) \ge 0,\]
so $G^+ \gtrsim G^-$ by Theorem~\ref{i-games-nice-2}.
\end{proof}
Note that if $H, K$ are two games in $\mathcal{I}$, both left-equivalent to $G$, then $H \approx_+ K$ so $H \approx K$ by Theorem~\ref{i-games-nice-2}.
Therefore the $G^\pm$ of the previous theorem are determined uniquely up to equivalence.
\begin{definition}
If $G$ is a well-tempered game, we let $G^+$ and $G^-$ be the games of the previous theorem.  These are defined only up to equivalence.
We call $G^+$ the \emph{upside} of $G$ and $G^-$ the \emph{downside} of $G$, by analogy with the theory of loopy games.
\end{definition}
Note that $G^+$ and $G^-$ depend only on $G$ up to $\approx$, i.e., if $G \approx H$ then $G^\pm \approx H^\pm$.  Also, if $G \in \mathcal{I}$,
then $G^+$ and $G^-$ can be taken to be $G$.

Let $\mathcal{W} = \mathcal{W}_\mathbb{Z}$.  For $\mathcal{G}$ a set of games, we use $\mathbf{G}$, $\mathbf{G}^+$, and $\mathbf{G}^-$
to denote the quotients $\mathcal{G}/\approx$, $\mathcal{G}/\approx_+$ and $\mathcal{G}/\approx_-$.  These have natural structures as posets,
coming from the relations $\lesssim$ and $\lesssim_\pm$.  If $\mathcal{G}$ is closed under addition,
then these three quotients will also be monoids.

Using upsides and downsides,
we get a fairly complete description of $\mathbf{W}$ and $\mathbf{W}^\pm$ in terms of $\mathbf{I}$.
\begin{theorem}\label{structure}
The posets $\mathbf{I}$, $\mathbf{I}^+$, and $\mathbf{I}^-$ are all equal, $\mathbf{I}$ is a partially-ordered abelian group, and the inclusions
\begin{equation} \mathbf{I} = \mathbf{I}^\pm \hookrightarrow \mathbf{W}^\pm \end{equation}
are isomorphisms of partially ordered monoids.  In particular, $\mathbf{W}^+$ and $\mathbf{W}^-$ are groups.  Also, let
\[ \mathbf{I}_2 = \{(x,y) \in \mathbf{I} \times \mathbf{I}~:~ x \ge y\}.\]
Make $\mathbf{I}_2$ into a partially ordered monoid by taking the product ordering and componentwise addition.  Then the map
\[ \mathbf{W} \to \mathbf{I}_2\]
induced by $G \mapsto (G^+,G^-)$ is an isomorphism of partially ordered monoids.  Under this identification,
the negative of $(x,y)$ is $(-y,-x)$, and the image of
\[ \mathbf{I} \hookrightarrow \mathbf{W} \to \mathbf{I}_2\]
is the diagonal $\{(x,x)~:~ x \in \mathbf{I}\}$.
\end{theorem}
\begin{proof}
The fact that $\mathbf{I}$, $\mathbf{I}^+$, and $\mathbf{I}^-$ are equal stems from the fact that $\lesssim_-$, $\lesssim_+$, and $\lesssim$
are all the same relation on $\mathcal{I}$, by Theorem~\ref{i-games-nice-2}.  The fact that $\mathcal{I}$ is a group follows from the fact
that if $G \in \mathcal{I}$, then $-G \in \mathcal{I}$ and $G + (-G) \approx 0$ by Theorem~\ref{i-games-nice-2}.  The inclusion
$\mathbf{I}^+ \to \mathbf{W}^+$ (say) is surjective by Theorem~\ref{sidling}.  It is injective and strictly order-preserving, trivially.

The remaining statements are equivalent to the following claims:
\begin{itemize}
\item If $G \in \mathcal{W}$, then $G^+ \gtrsim G^-$.  This is part of Theorem~\ref{sidling}.
\item If $G \approx H$, then $G^+ \approx H^+$ and $G^- \approx H^-$.  More generally,
\[ G \gtrsim H \iff G^+ \gtrsim H^+ \wedge G^- \gtrsim H^-.\]
This follows from the facts that
\[ G \gtrsim_\pm H \iff G^\pm \gtrsim_\pm H^\pm \iff G^\pm \gtrsim H^pm.\]
The first $\iff$ is trivial and the second $\iff$ follows from Theorem~\ref{i-games-nice-2}.
\item If $G, H \in \mathcal{I}$ and $G \gtrsim H$, then there exists a game $K$ with $K^+ \approx G$ and $K^- \approx H$, i.e., $K \approx_+ G$
and $K \approx_- H$.  We defer the proof of this to Theorem~\ref{pairs-occur} in the next section.
\item If $G, H$ are games, then
\[ (G + H)^\pm \approx G^\pm + H^\pm.\]
Since $\mathcal{I}$ is closed under addition, this amounts to showing that if $G \approx_\pm G^\pm$ and $H \approx_\pm H^\pm$, then $G + H \approx_\pm G^\pm + H^\pm$,
which follows from Theorem~\ref{obvious-compatibles}.
\item If $G$ is a game, then $-G^\pm = (-G)^\mp$.  This similarly follows from Theorem~\ref{obvious-compatibles}.
\item If $G \in \mathcal{I}$, then $G^+ \approx G^-$, which is clear by taking $G^+ = G^- = G$.
\item If $G^+ \approx G^-$, then $G$ is equivalent to a game in $\mathcal{I}$.  But if $G^+ \approx G^-$, then $G \approx_- G^- \approx_- G^+$, so
$G \approx_- G^+$.  Since $G \approx_+ G^+$ by definition of $G^+$, we have $G \approx G^+ \in \mathcal{I}$.
\end{itemize}
\end{proof}

\begin{corollary}\label{i-stands-for-invertible}
If $G$ is a game, the following are equivalent:
\begin{description}
\item[(a)] $G$ is invertible, in the sense that $G + H \approx 0$ for some $H \in \mathcal{W}$.
\item[(b)] $G \approx K$ for some $K \in \mathcal{I}$.
\item[(c)] $G - G \approx 0$
\end{description}
\end{corollary}
\begin{proof}
The direction $(c) \implies (a)$ is trivial, and $(b) \implies (c)$ was part of Theorem~\ref{i-games-nice-2}.  The direction $(a) \implies (b)$ follows
from the general fact that if $\mathfrak{G}$ is a partially-ordered abelian group and we construct the monoid
\[ \mathfrak{G}_2 = \{(x,y) \in \mathfrak{G} \times \mathfrak{G} : x \ge y\},\]
then the invertible elements of $\mathfrak{G}_2$ are exactly the diagonal elements $(x,x)$.
\end{proof}

To describe the structure of $\mathbf{W}_\mathbb{Z}$, it remains to determine the structure of $\mathbf{I}$.
In \S \ref{sec:psi} we will relate $\mathbf{I}$ to the group $\mathbf{Pg}$ of partizan normal-play games.

\section{Other Disjunctive Operations}\label{sec:distortions}
So far we have focused on the operation of disjunctive addition $G + H$.  In this section, we show that other generalized disjunctive operations
behave like addition.  For example they are compatible with equivalence, and they preserve membership in $\mathcal{I}$.

\subsection{Easy Inductions}\label{basic-distort}
The following facts have easy inductive proofs, which we leave as exercises to the reader:
\begin{description}
\item[(a)] If $f : S \to T$ is order-preserving and $G$ is an $S$-valued game, then $\rout(\tilde{f}(G)) = f(\rout(G))$ and similarly for $\lout(\cdot)$, $\rfout(\cdot)$,
and $\lfout(\cdot)$.
\item[(b)] If $f, g : S_1 \times \cdots \times S_n \to T$ are two order preserving functions and $f(x_1,\ldots,x_n) \le g(x_1,\ldots,x_n)$ for all $x_i$, then
\[ \tilde{f}(G_1,\ldots,G_n) \lesssim \tilde{g}(G_1,\ldots,G_n).\]
\item[(c)] If $f : S_1 \times \cdots \times S_n \to T$ is a function (not necessarily order-preserving), $\sigma$ is a permutation of $\{1,\ldots,n\}$, and
$g : S_{\sigma(1)} \times \cdots \times S_{\sigma(n)} \to T$ is the function given by
\[ g(x_{\sigma(1)},\ldots,x_{\sigma(n)}) = f(x_1,\ldots,x_n),\]
then
\[ \tilde{g}(G_{\sigma(1)},\ldots,G_{\sigma(n)}) = \tilde{f}(G_1,\ldots,G_n)\]
\item[(d)] If $f : S_1 \times \cdots \times S_n \to T$ is a function, $g : R_1 \times \cdots \times R_m \to S_j$
is another function, and
\[ h : S_1 \times \cdots \times S_{j-1} \times R_1 \times \cdots \times R_m \times S_{j+1} \times \cdots \times S_n \to T\]
is the function given by
\[ h(x_1,\ldots,x_{j-1},y_1,\ldots,y_m,x_{j+1},\ldots,x_n) = f(x_1,\ldots,x_{j-1},g(y_1,\ldots,y_m),x_{j+1},\ldots,x_n),\]
then
\[ \tilde{h}(G_1,\ldots,G_{j-1},H_1,\ldots,H_m,G_{j+1},\ldots,G_n) = \tilde{f}(G_1,\ldots,G_{j-1},\tilde{g}(H_1,\ldots,H_m),G_{j+1},\ldots,G_n).\]
In particular, if $m = n = 1$ we see that $\widetilde{f \circ g} = \tilde{f} \circ \tilde{g}$.
\item[(e)] If $f : S_1 \times \cdots \times S_n \times T$ is a function and $g : (-S_1) \times \cdots \times (-S_n) \to (-T)$ is the function
\[ g(x_1,\ldots,x_n) = -f(-x_1,\ldots,-x_n),\]
then
\[ \tilde{g}(G_1,\ldots,G_n) = -\tilde{f}(-G_1,\ldots,-G_n).\]
\item[(f)] If $g : \mathbb{Z} \to \mathbb{Z}$ is the constant function 0, then $\tilde{g}(G) \approx 0$ for $\pi(G) = 0$ and $\tilde{g}(G) \approx \{0|0\}$
for $\pi(G) = 1$.
\item[(g)] If $f : S_1 \times \cdots \times S_n$ is a function and $*$ is the game $\{0|0\}$, then
\[ \tilde{f}(G_1 + *,\ldots,G_n) = \tilde{f}(G_1,\ldots,G_n) + *\]
\end{description}

\subsection{Equivalence and Generalized Disjunctive Operations}
The following theorem generalizes to $f$ of higher or lower arities, but we only prove the case where $f$ takes two arguments, for notational simplicity.
\begin{theorem}\label{generalized-compatibility}
Let $S_1, S_2 \subseteq \mathbb{Z}$ and $f : S_1 \times S_2 \to \mathbb{Z}$ be order preserving.  Let $G_1, H_1$ be $S_1$-valued games,
and $G_2, H_2$ be $S_2$-valued games.  Let $\Box$ be one of the relations $\gtrsim, \lesssim, \approx, \gtrsim_\pm, \lesssim_\pm, \approx_\pm$.
Suppose that $G_i \Box H_i$ for $i = 1, 2$.  Then
\[ \tilde{f}(G_1,G_2) \Box \tilde{f}(H_1,H_2),\]
where $\tilde{f}$ is the extension of $f$ to games.
\end{theorem}
\begin{proof}
Because all nine relations can be defined in terms of $\gtrsim_+$ and $\gtrsim_-$, we can assume $\Box$ is $\gtrsim_\pm$.  By symmetry,
we can take $\Box = \gtrsim_+$.

Suppose we can show that for $G, H, X$ appropriately-valued,
\[ G \gtrsim_+ H \implies \tilde{f}(G,X) \gtrsim_+ \tilde{f}(H,X)\]
Then by symmetry (interchanging the two arguments of $f$),
\[ G \gtrsim_+ H \implies \tilde{f}(X,G) \gtrsim_+ \tilde{f}(X,H).\]
Combining these, we get
\[ \tilde{f}(G_1,H_1) \gtrsim_+ \tilde{f}(G_1,H_2) \gtrsim_+ \tilde{f}(G_2,H_2)\]
as desired.

We are reduced to showing that if $f : S_1 \times S_2 \to \mathbb{Z}$ is order-preserving and $G \gtrsim_+ H$, then
\[ \tilde{f}(G,X) \gtrsim_+ \tilde{f}(H,X).\]
Since only finitely many integers occur within $G, H, X$, we can assume without loss of generality that $S_1$ and $S_2$ are finite,
restricting the domain of $f$ if necessary.

We need to show that for any game $Y$,
\[ \lfout(\tilde{f}(G,X) + Y) \ge \lfout(\tilde{f}(H,X) + Y).\]
Suppose not.  Then there is some $n$ such that
\[ \lfout(\tilde{f}(G,X) + Y) < n\]
\[ \lfout(\tilde{f}(H,X) + Y) \ge n.\]
Let $\delta_m : \mathbb{Z} \to \{0,1\}$ be the order-preserving function that sends $x$ to $1$ if $x \ge m$ and $x$ to $0$ if $x < m$.
By Lemma~\ref{technical} applied to the function $g(w,x,y) = \delta_n(f(w,x) + y)$,
there is a function $h : S_2 \times \mathbb{Z} \to \mathbb{Z}$ such that
\[ \delta_n(f(w,x) + y) = \delta_0(w + h(x,y))\]
for all $(w,x,y) \in S_1 \times S_2 \times \mathbb{Z}$.  Then by several applications of (c) and (d) of \S \ref{basic-distort},
\[ \tilde{\delta}_n(\tilde{f}(G,X) + Y) = \tilde{\delta}_0(G + \tilde{h}(X,Y))\]
\[ \tilde{\delta}_n(\tilde{f}(H,X) + Y) = \tilde{\delta}_0(H + \tilde{h}(X,Y)).\]
But by (a) of \S \ref{basic-distort}, it follows that
\[ 0 = \delta_n(\lfout(\tilde{f}(G,X) + Y)) = \delta_0(\lfout(G + \tilde{h}(X,Y)))\]
\[ 1 = \delta_n(\lfout(\tilde{f}(H,X) + Y)) = \delta_0(\lfout(H + \tilde{h}(X,Y))).\]
Then
\[ \lfout(G + \tilde{h}(X,Y)) < \lfout(H + \tilde{h}(X,Y)),\]
contradicting $G \gtrsim_+ H$.
\end{proof}
\begin{lemma}\label{technical}
Let $S_1, S_2, S_3$ be subsets of $\mathbb{Z}$, with $S_1$ finite.  Let $g : S_1 \times S_2 \times S_3 \to \{0,1\}$ be order-preserving.
Then there is a function $h : S_2 \times S_3 \to \{0,1\}$ such that for every $(x,y,z) \in S_1 \times S_2 \times S_3$,
\[ g(x,y,z) = 1 \iff x + h(y,z) \ge 0\]
\end{lemma}
\begin{proof}
Let $M$ be an integer greater than every element of $S_1$.  Replacing $S_1$ with $S_1 \cup \{M\}$ and extending $g(x,y,z)$ by
$g(M,y,z) = 1$ for all $y,z$, we can assume that for every $y,z$, there is some $x \in S_1$ such that $g(x,y,z) = 1$.
Let $h(y,z)$ be the negative of the smallest such $x$.  Then $g(x,y,z) = 1$ if and only if $x \ge -h(y,z)$, as desired.
\end{proof}

From the theorem (and its generalization to higher arities), we see that if $f : S_1 \times \cdots \times S_n \to T$ is order-preserving,
then $\tilde{f}$ induces well-defined functions
\[ \mathbf{W}_{S_1} \times \cdots \times \mathbf{W}_{S_n} \to \mathbf{W}_{T} \]
\[ \mathbf{W}^+_{S_1} \times \cdots \times \mathbf{W}^+_{S_n} \to \mathbf{W}^+_{T} \]
\[ \mathbf{W}^-_{S_1} \times \cdots \times \mathbf{W}^-_{S_n} \to \mathbf{W}^-_{T} \]
\subsection{Invertible Games and Generalized Disjunctive Operations}
Our next goal is to show that $\mathcal{I}$ is closed under generalized disjunctive operations.  In particular, letting
$\mathcal{I}_S = \mathcal{I} \cap \mathcal{W}_S$, there is a well-defined map
\[ \mathbf{I}_{S_1} \times \cdots \times \mathbf{I}_{S_n} \to \mathbf{I}_{T}.\]
This map characterizes the full map $\mathbf{W}_{S_1} \times \cdots \times \mathbf{W}_{S_n} \to \mathbf{W}_{T}$, since by Theorem~\ref{generalized-compatibility},
\[ \tilde{f}(G_1,\ldots, G_n)^\pm \approx \tilde{f}(G_1^\pm,\ldots,G_n^\pm)\]
assuming that the right hand side lies in $\mathcal{I}$, as we shall prove.

We first prove something that is almost the desired result:
\begin{lemma}\label{almost-there}
Let $f : S_1 \times \cdots \times S_n \to \mathbb{Z}$ be order-preserving, and for each $i$, let $G_i \in \mathcal{I}_{S_i}$.  Then
$\tilde{f}(G_1,\ldots,G_n)$ is equivalent to a game in $\mathcal{I}$.
\end{lemma}
\begin{proof}
Since each $G_i$ has finitely many subgames, we can assume without loss of generality that every $S_i$ is finite.

Let $M$ be an integer greater in magnitude than every element of the range of $f$ (which is finite, because we made the $S_i$ finite).
Let $\kappa : \mathbb{Z} \to \mathbb{Z}$ be the function
\[ \kappa(x) = \begin{cases} 2M & x > 0 \\ 0 & x \le 0 \end{cases}.\]
Then since $G_i - G_i \approx 0$,
\[ \tilde{\kappa}(G_i - G_i) \approx \tilde{\kappa}(0) = 0\]
for each $i$, by Theorem~\ref{generalized-compatibility}.
It follows that
\[ \tilde{\kappa}(G_1 - G_1) + \cdots + \tilde{\kappa}(G_n - G_n) \approx 0.\]

Consider the two functions $S_1 \times (-S_1) \times  \cdots \times S_n \times (-S_n) \to \mathbb{Z}$ given by
\begin{align*} h_1(x_1,x_2,\ldots,x_{2n}) &= f(x_1,x_3,\ldots,x_{2n-1}) - f(-x_2,-x_4,\ldots,-x_{2n}) \\
h_2(x_1,x_2,\ldots,x_{2n}) &= \kappa(x_1 + x_2) + \kappa(x_3 + x_4) + \cdots + \kappa(x_{2n-1} + x_{2n}).
\end{align*}
It is straightforward to check that $h_1 \le h_2$.  Therefore, by (b) of \S \ref{basic-distort},
\[ \tilde{h}_1(G_1,-G_1,\ldots,G_n,-G_n) \lesssim \tilde{h}_2(G_1,-G_1,\ldots,G_n,-G_n)\]
However, $\tilde{h}_1(G_1,-G_1,\ldots,G_n,-G_n)$ is nothing but
\[ \tilde{f}(G_1,G_2,\ldots,G_n) - \tilde{f}(G_1,\ldots,G_n)\]
and $\tilde{h}_2(G_1,-G_1,\ldots,G_n,-G_n)$ is nothing but
\[ \tilde{\kappa}(G_1 - G_1) + \cdots + \tilde{\kappa}(G_n - G_n) \approx 0 + \cdots + 0 = 0,\]
by several applications of (c), (d), and (e) of \S \ref{basic-distort}.
Thus
\[ \tilde{f}(G_1,G_2,\ldots,G_n) - \tilde{f}(G_1,\ldots,G_n) \lesssim 0\]
Negating both sides,
\[ \tilde{f}(G_1,G_2,\ldots,G_n) - \tilde{f}(G_1,\ldots,G_n) \gtrsim 0.\]
Therefore $\tilde{f}(G_1,\ldots,G_n)$ is invertible, so by Corollary~\ref{i-stands-for-invertible}, it is equivalent to a game in $\mathcal{I}$.
\end{proof}

Our desired result follows easily:
\begin{theorem}\label{i-games-closed}
If $f : S_1 \times \cdots \times S_n \to \mathbb{Z}$ is order-preserving and $G_i \in \mathcal{I}_{S_i}$ for each $i$,
then $\tilde{f}(G_1,\ldots,G_n) \in \mathcal{I}$.
\end{theorem}
\begin{proof}
By the previous lemma, every subgame of $\tilde{f}(G_1,\ldots,G_n)$ is equivalent to a game in $\mathcal{I}$.  It therefore suffices
to show that if $H$ is a game and every subgame of $H$ (including $H$) is invertible, then $H \in \mathcal{I}$.  If $H'$ is any
even-tempered subgame of $H$, then $\lout(H') - \rout(H') \ge 0$ because $H' \approx K$ for some $K \in \mathcal{I}$, and $\lout(K) - \rout(K) \ge 0$
by definition of $\mathcal{I}$.
\end{proof}



Another corollary of Theorem~\ref{generalized-compatibility} is the following:
\begin{corollary}\label{same-size}
Let $S$ and $T$ be finite subsets of $\mathbb{Z}$, with $|S| = |T|$.  Let $f$ be the order-preserving bijection
from $S$ to $T$.  Then $\tilde{f} : \mathcal{W}_S \to \mathcal{W}_T$ induces a bijection
\[ \mathbf{W}_S \to \mathbf{W}_T\]
\end{corollary}
\begin{proof}
Let $g : T \to S$ be the order-preserving function $f^{-1}$.  By (d) of \S \ref{basic-distort}, $\tilde{f}$ and $\tilde{g}$ are inverses of each other.
By Theorem~\ref{generalized-compatibility}, they induce well-defined order-preserving maps from $\mathbf{W}_S \to \mathbf{W}_T$ and vice-versa.
These two induced maps must be inverses of each other.
\end{proof}
Because of this corollary, the study of $S$-valued games is equivalent to the study of $\{1,\ldots,n\}$-valued games,
for $n = |S|$.

\subsection{The Possible Upsides and Downsides}
We now complete the proof of Theorem~\ref{structure}:
\begin{theorem}\label{pairs-occur}
Let $G, H$ be games in $\mathcal{I}$, with $G \gtrsim H$.  Then there is a $K$ such that $K \approx_+ G$ and $K \approx_- H$, i.e.,
$K^+ \approx G$ and $K^- \approx H$.  Moreover, if $G$ and $H$ are in $\mathcal{I}_S$, then $K$ can be chosen to be in $\mathcal{W}_S$.
Therefore the map
\[ \mathbf{W}_S \to \{(x,y) ~:~ x,y \in \mathbf{I}_S,~x \ge y\}\]
induced by
\[ G \mapsto (G^+,G^-)\]
is surjective.
\end{theorem}
\begin{proof}
For $n \ge 0$, let $Q_n = \{\{0|0\}|\{n|n\}\}$.  One can check directly that $Q_n \approx_+ n$ and $Q_n \approx_- 0$ by Lemma~\ref{gap-rule}.

Let $\Delta$ be the game $G - H$.  Then $\Delta \in \mathcal{I}$ and $\Delta \gtrsim 0$.  In particular, $\Delta$ is even.
Let $\Delta'$ be the game obtained from $\Delta$ by replacing every integer subgame $n$ with
\[ \begin{cases}
n & \text{ if } n \le 0 \\
Q_n & \text{ if } n > 0
\end{cases}\]
Let $f : \mathbb{Z} \to \mathbb{Z}$ be the function $f(n) = \min(0,n)$.
Then $\Delta' \approx_+ \Delta$ and $\Delta' \approx_- \tilde{f}(\Delta)$ by an easy induction using the fact
that $Q_n \approx_+ n$ and $Q_n \approx_- 0$.  Let $g : \mathbb{Z} \to \mathbb{Z}$ be the constant function 0.  Since $g \ge f$ and $\Delta \gtrsim 0$,
\[ 0 \approx \tilde{g}(\Delta) \gtrsim \tilde{f}(\Delta) \gtrsim \tilde{f}(0) = 0,\]
by (f) and (b) of \S \ref{basic-distort} and Theorem~\ref{generalized-compatibility}.
Therefore $\tilde{f}(\Delta) \approx 0$.
Since $\Delta' \approx_+ n$ and $\Delta' \approx_- \tilde{f}(\Delta) \approx 0$, we conclude that
\[ \Delta' + H \approx_+ \Delta + H = (G - H) + H \approx G\]
\[ \Delta' + H \approx_- 0 + H \approx H\]
Letting $K_0 = \Delta' + H$, $K_0$ has the desired upside and downside.  It remains to make $K_0$ take values in $S$, if both $G$ and $H$ do.

Suppose $G$ and $H$ take values in $S$. Let $r : \mathbb{Z} \to S$ be the function that sends
each $x \in \mathbb{Z}$ to the closest element of $S$, taking the lesser value in case of a tie.  Clearly $r$ is an order-preserving function
that agrees with the identity on $S$.  Thus $\tilde{r}(G) = G$ and $\tilde{r}(H) = H$, and $\tilde{r}(K_0)$ is $S$-valued.
By Theorem~\ref{generalized-compatibility},
\[ \tilde{r}(K_0) \approx_+ \tilde{r}(G) = G\]
and
\[ \tilde{r}(K_0) \approx_- \tilde{r}(H) = H,\]
Taking $K = \tilde{r}(K_0)$, we get the desired game.
\end{proof}

If $G, H \in \mathcal{I}$ and $G \gtrsim H$, we let $G\& H$ be the game of the theorem, whose upside is $G$ and whose downside is $H$.
This is well-defined up to equivalence.  Our notation is based on the analogy with the loopy partizan theory.

\subsection{Even and Odd}
Let $*$ be the game $\{0|0\}$.  If $G$ is an $S$-valued game, then so is $G + *$.  Indeed, letting $f : S \times \{0\} \to S$ be addition,
$G + * = \tilde{f}(G,*) \in \mathcal{W}_S$.  Now clearly $* \in \mathcal{I}$, so if $G \in \mathcal{I}_S$, then $G + * \in \mathcal{I}_S$.
As $-* = *$, the game $*$ is its own inverse.  Therefore, the map $G \mapsto G + *$ induces involutions on $\mathbf{W}_S$ and $\mathbf{I}_S$
for any $S$.  These involutions exchange the even-tempered and odd-tempered games, establishing a bijection between the even and odd elements in
each of $\mathbf{W}_S$ and $\mathbf{I}_S$.  (Note that we can talk about the parity of an element of $\mathbf{W}_S$ by Theorem~\ref{parity-incomparable}.)

\begin{definition}
Let $G$ be a game.  The \emph{even projection of $G$}, denoted $\even(G)$, is $G$ if $G$ is even-tempered, and $G + *$ otherwise.
\end{definition}

This has a number of simple properties:
\begin{theorem}
\begin{description}
Let $G, H$ be games.
\item[(a)]
If $G$ is $S$-valued, then so is $\even(G)$.
\item[(b)]
If $G \in \mathcal{I}$, then $\even(G) \in \mathcal{I}$.
\item[(c)]
$\even(G)$ is even-tempered and $\even(\even(G)) = \even(G)$.
\item[(d)]
$\even(-G) = -\even(G)$.
\item[(e)]
$G \gtrsim H$ if and only if $\even(G) \gtrsim \even(H)$ and $\pi(G) = \pi(H)$.  In particular, if $G \approx H$ then $\even(G) \approx \even(H)$,
so $\even$ induces well-defined maps on $\mathbf{W}_S$ and $\mathbf{I}_S$ for any $S$.
\item[(f)]
If $\tilde{f}$ is a generalized disjunctive operation, then
\[ \tilde{f}(\even(G_1),\ldots,\even(G_n)) = \even(\tilde{f}(G_1,\ldots,G_n)).\]
In particular, for the case of addition,
\[ \even(G + H) \approx \even(G) + \even(H).\]
\item[(g)]
$\even(G^+) \approx \even(G)^+$ and $\even(G^-) \approx \even(G)^-$.
\end{description}
\end{theorem}
\begin{proof}
We already proved (a) and (b).  (c) is trivial.  (d) says that $-G = -G$ when $G$ is even-tempered, and that $*-G = -(*+G)$ when $G$ is odd-tempered,
both of which are trivial.  The remaining three statements are proven as follows:
\begin{description}
\item[(e)]  Suppose $G \gtrsim H$.  Then by Theorem~\ref{parity-incomparable}, $\pi(G) = \pi(H)$.  It remains to see that if $\pi(G) = \pi(H)$,
then $G \gtrsim H \iff \even(G) \gtrsim \even(H)$.  Let $\epsilon$ be $0$ if $\pi(G) = \pi(H) = 0$, and $*$ if $\pi(G) = \pi(H) = 1$.
Then $\even(G) = G + \epsilon$ and $\even(H) = H + \epsilon$.  But clearly
\[ G \gtrsim H \iff G + \epsilon \gtrsim H + \epsilon,\]
since $\epsilon$ is invertible.
\item[(f)]
It follows by repeated applications of (g) of \S \ref{basic-distort} that
\[ \tilde{f}(\even(G_1),\ldots,\even(G_n)) = \tilde{f}(G_1,\ldots,G_n) + * + \cdots + *,\]
where the number of $*$'s is the number of $G_i$ which are odd-tempered.  Thus
\begin{equation} \even(\tilde{f}(\even(G_1),\ldots,\even(G_n))) \approx \even(\tilde{f}(G_1,\ldots,G_n) + * + \cdots + *).\label{step-one}\end{equation}
Now if $X$ is any game, then $\even(X + *) \approx X$
since $\even(X + *)$ is either $X$ or $X + * + * \approx X$.
Thus
\begin{equation} \even(\tilde{f}(G_1,\ldots,G_n) + * + \cdots + *) \approx \even(\tilde{f}(G_1,\ldots,G_n)).\label{step-two}\end{equation}
But since every $\even(G_i)$ is even-tempered, $\tilde{f}(\even(G_1),\ldots,\even(G_n))$ is also even-tempered so
\begin{equation} \tilde{f}(\even(G_1,\ldots,\even(G_n))) = \even(\tilde{f}(\even(G_1),\ldots,\even(G_n))).\label{step-three}\end{equation}
Combining equations (\ref{step-one}-\ref{step-three}) gives the desired result.
\item[(g)] If $G$, $G^+$, and $G^-$ are even-tempered, this says that $G^+ \approx G^+$ and $G^- \approx G^-$.  Otherwise, all three are odd-tempered, and this says
\[ * + G^\pm \approx (G + *)^\pm\]
which follows because $(G + H)^\pm \approx G^\pm + H^\pm$ for any $H$ (part of Theorem~\ref{structure}) and $*^\pm \approx *$ since $* \in \mathcal{I}$.
\end{description}
\end{proof}

We summarize the results of \S \ref{sec:distortions} in the following analog of Theorem~\ref{structure}, whose proof we leave as an exercise to the reader.
\begin{theorem}\label{structure-small}
Let $S \subseteq \mathbb{Z}$.  Let $\mathcal{I}_S^0$ denote the even-tempered elements of $\mathcal{I}_S$.
Then we have a natural isomorphism of posets
\[ \mathbf{W}_S \cong \{(x,y) \in \mathbf{I}_S^0 : x \ge y\} \times \mathbb{Z}/2\mathbb{Z}\]
induced by
\[ G \mapsto (\even(G)^+, \even(G)^-, \pi(G)) \]
If $f : S^k \to S$ is order-preserving, then the extension to games $\tilde{f} : \mathcal{W}_S^k \to \mathcal{W}_S$ induces
maps
\[ \tilde{f} : \mathbf{W}_S^k \to \mathbf{W}_S\]
\[ \tilde{f} : (\mathbf{I}_S^0)^k \to \mathbf{I}_S^0.\]
The restriction to $\mathbf{I}_S^0$ determines the behavior on $\mathbf{W}_S$ as follows:
\[ \even(\tilde{f}(G_1,\ldots,G_k))^+ \approx \tilde{f}(\even(G_1)^+,\ldots,\even(G_k)^+)\]
\[ \even(\tilde{f}(G_1,\ldots,G_k))^- \approx \tilde{f}(\even(G_1)^-,\ldots,\even(G_k)^-)\]
\[ \pi(\tilde{f}(G_1,\ldots,G_k)) \equiv \pi(G_1) + \cdots + \pi(G_k) \quad (\modforreal 2)\]
\end{theorem}
Later when we consider $\{0,1\}$-valued games, this will allow us to restrict our attention to $\mathbf{I}^0_{\{0,1\}}$,
which has only eight elements, rather than considering $\mathbf{W}_{\{0,1\}}$, which has seventy elements.

\section{Reduction to Partizan Games}\label{sec:psi}
Returning to the general case of $\mathbb{Z}$-valued games and addition, the goal of this section is to give
an explicit description of $\mathbf{I}_{\mathbb{Z}}$ (or equivalently $\mathbf{I}_{\mathbb{Z}}^0$) in terms of the
classical combinatorial game theory of disjunctive sums of partizan games.

Earlier we defined $\mathcal{J} \subseteq \mathcal{I}$ to be the set of games $G$ with $\gap_0(G) = 0$ and $\gap_1(G) \le -2$.
By Theorem~\ref{gap-sums}, $\mathcal{I}$ and $\mathcal{J}$ are closed under addition.  Since both consist of invertible games,
$\mathbf{I}$ and $\mathbf{J}$ are both subgroups of the monoid $\mathbf{W}$.  Now by (\ref{heat-gap}) of Lemma~\ref{early-lemma},
heating has the following effects on $\gap_0$ and $\gap_1$:
\[ \gap_0\left(\int^t G\right) = \gap_0(G)\]
\[ \gap_1\left(\int^t G\right) = \gap_1(G) - 2t.\]
Therefore, $\mathcal{I}$ is preserved by heating, $\mathcal{J}$ is preserved by heating by $t \ge 0$, and if $G$ is any game
in $\mathcal{I}$, then $\int^t G \in \mathcal{J}$ for $t \gg 0$.

Because $\mathcal{I}$ is preserved by heating, and heating is compatible with $\approx_\pm$ by Lemma~\ref{obvious-compatibles},
\[ \int^t G \approx_\pm \int^t (G^\pm) \in \mathcal{I}\]
so
\[ \left(\int^t G\right)^\pm \approx \int^t (G^\pm).\]
Thus heating is compatible with upsides and downsides.
Also, heating induces monoid automorphisms on $\mathbf{W}$ and $\mathbf{I}$.
Note however that heating is not compatible with the map $G \mapsto \even(G)$ of the previous section, because $\int^t * = \{t|-t\} \not\approx *$.

\subsection{The Class \textbf{K}}
We would like to define a class of games $\mathcal{K}$ that has the same relation to $\mathcal{W}$ that $\mathcal{J}$ has to $\mathcal{I}$.

Let $\mathcal{K}'$ be the class of games $G$ such that both $G^+$ and $G^-$ can be chosen to lie in $\mathcal{J}$.  Let $\mathcal{K}$
be the class of games $G$ such that every subgame of $G$ (including $G$ itself) lies in $\mathcal{K}'$.  Equivalently, $G \in \mathcal{K}$
if every option of $G$ is in $\mathcal{K}$ and $G \in \mathcal{K}'$.
\begin{theorem}\label{k-non-void}
~
\begin{description}
\item[(a)]
$\mathcal{J} = \mathcal{K} \cap \mathcal{I}$.  In particular, $\mathcal{J} \subseteq \mathcal{K}$.
\item[(b)]
$\mathcal{K}$ is closed under addition and negation.
\item[(c)]
If $G$ is a game, then $\int^t G \in \mathcal{K}$ for $t \gg 0$.  Also, if $G \in \mathcal{K}$ then $\int^t G \in \mathcal{K}$ for $t \ge 0$.
\item[(d)]
If $S \subseteq \mathbb{Z}$, and $\max(S) - \min(S) \le 2$, then $\mathcal{W}_S \subseteq \mathcal{K}$.
\end{description}
\end{theorem}
\begin{proof}
\begin{description}
\item[(a)]
If $G \in \mathcal{J}$, then certainly $G \in \mathcal{I}$.  Also, for every subgame $H$ of $G$, $H \in \mathcal{K}'$ because we can
take $H^+ = H^- = H \in \mathcal{J}$.  So $G \in \mathcal{K}$.

Conversely, suppose that $G \in \mathcal{I} \cap \mathcal{K}$.  If $H$ is any odd-tempered subgame of $G$, then $H \in \mathcal{K}' \cap \mathcal{I}$.
Therefore $H \approx H^+$ for some $H^+ \in \mathcal{J}$, and $\lout(H) - \rout(H) = \lout(H^+) - \rout(H^+) \ge -2$.  Since $H$ was arbitrary,
$\gap_1(G) \ge -2$, and $G \in \mathcal{J}$.
\item[(b)]
Since $\mathcal{J}$ is closed under addition, and addition is compatible with taking upsides and downsides, $\mathcal{K}'$ is also closed under addition.
Now let $G, H \in \mathcal{K}$.  By induction, every option of $G + H$ is in $\mathcal{K}$.  So $G + H \in \mathcal{K}$ as long
as $G + H \in \mathcal{K}'$, which follows from the fact that both $G$ and $H$ are in $\mathcal{K}'$.
The definitions of $\mathcal{I}$, $\mathcal{J}$, and $\mathcal{K}$ are symmetric, not distinguishing the two players, so by symmetry all three classes are closed
under negation.
\item[(c)]
The following fact has an easy inductive proof, which we leave as an exercise to the reader:
\begin{fact}
If $G$ is a game, and $H$ is a subgame of $\int^t G$, then $H = nt + \int^t G'$ for some subgame $G'$ of $G$ and some $n \in \mathbb{Z}$.
\end{fact}
Now if $H$ is any game, then $H^+, H^- \in \mathcal{I}$.  As noted above, if we heat games sufficiently, they will lie in $\mathcal{J}$.  So for any
$H$, we have
\[ \left(\int^t H\right)^\pm = \int^t (H^\pm) \in \mathcal{J}\]
for $t \gg 0$.  Thus $\int^t H \in \mathcal{K}'$ for $t \gg 0$.

Applying this to every subgame of $G$, we see that for $t \gg 0$, every subgame $G'$ of $G$ satisfies $\int^t G' \in \mathcal{K}'$.  Now let $H$ be any subgame
of $\int^t G$.  By the Fact, $H = nt + \int^t G'$ for some $n \in \mathbb{Z}$ and some subgame $G'$ of $G$.  Then $nt \in \mathbb{Z} \subset \mathcal{J} \subset
\mathcal{K}'$ and $\int^t G' \in \mathcal{K}'$.  Then since $\mathcal{K}'$ is closed under addition, $H = nt + \int^t G' \in \mathcal{K}'$.  Since
$H$ was an arbitrary subgame of $\int^t G$, it follows that $\int^t G \in \mathcal{K}$.

A similar argument shows the second claim of (c).
\item[(d)]
First note that if $G$ is any $S$-valued game, then $\lout(G), \rout(G) \in S$, so
\[ |\lout(G) - \rout(G)| \le |\max(S) - \min(S)| \le 2.\]
In particular, $\gap_1(G) \le 2$.  It follows that $\mathcal{I}_S \subseteq \mathcal{J}$.
Then if $G$ is in $\mathcal{W}_S$ and $H$ is any subgame of $G$, $H^+$ and $H^-$ can be chosen to lie in $\mathcal{I}_S \subseteq \mathcal{J}$.
Therefore $G \in \mathcal{K}$.
\end{description}
\end{proof}

\subsection{From Scoring Games to Partizan Games}
We now turn to standard (normal-play, short, non-loopy) partizan games.
\begin{definition}
If $G^L_1,\ldots, G^R_1,\ldots$ are partizan games, let
\begin{equation} \{G^L_1,\ldots|G^R_1,\ldots\}_+ \label{plus-bars}\end{equation}
\begin{equation} \{G^L_1,\ldots|G^R_1,\ldots\}_+ \label{minus-bars}\end{equation}
denote the usual partizan game $\{G^L_1,\ldots|G^R_1,\ldots\}$ unless there is at least one integer $n$ with
\[ G^L_i \lhd n \lhd G^R_j\]
for every $i, j$, in which case we let (\ref{plus-bars}) denote the largest such $n$
and (\ref{minus-bars}) denote the smallest such $n$.
\end{definition}
These operations show up in the definition of atomic weight in the theory of all-small partizan games, suggesting a possible
connection to the present theory.

One convenient property of $\{\cdot|\cdot\}_\pm$ is the following fact:
\begin{lemma}\label{automatic-integer-avoidance}
If $n$ is an integer and $\{G^L\}, \{G^R\}$ are sets of partizan games, then
\[ n + \{G^L|G^R\}_+ = \{n + G^L|n + G^R\}_+\]
and
\[ n + \{G^L|G^R\}_- = \{n + G^L|n + G^R\}_-\]
\end{lemma}
\begin{proof}
When the caveat that distinguishes $\{G^L|G^R\}_\pm$ from $\{G^L|G^R\}$ doesn't apply, this is just the usual integer avoidance theorem,
since $\{G^L|G^R\}$ does not equal an integer.  In the other case, the largest (smallest) integer which lies between the $n + G^L$ and the
$n + G^R$ is $n$ more than the largest (smallest) integer which lies between the $G^L$ and the $G^R$, so the result is clear.
\end{proof}
\begin{lemma}\label{hopefully-this-is-true}
Let $G_- = \{G^L|G^R\}_-$ and $H_- = \{H^L|H^R\}_-$.  Then
\begin{equation} G_- + H_- = \{G^L + H_-, G_- + H^L| G^R + H_-, G_- + H^R\}_- \label{ugh}\end{equation}
\end{lemma}
\begin{proof}
If $\{K^L\}$ and $\{K^R\}$ are sets of partizan games, then for $n \gg 0$,
\[ \{K^L|K^R\}_- + n = \{K^L + n|K^R + n\}_- = \{K^L + n|K^R + n\}\]
by Lemma~\ref{automatic-integer-avoidance} and the simplicity rule.
Therefore, for $n \gg 0$,
\begin{align*} (G_- + n) + (H_- + n) = & \{G^L + n|G^R + n\} + \{H^L + n| H^R + n\} \\
 = & \{(G^L+n)+(H_-+n),~(G_-+n)+(H^L+n)\\ &|(G^R+n)+(H_-+n),~(G_-+n)+(H^R+n)\} \\
 = & \{G^L + H_- + 2n,~G_- + H^L + 2n \\ &|G^R+H_- + 2n,~G_- + H^R + 2n\}. \end{align*}
Since $2n \gg 0$, this equals
\begin{align*} &\{G^L + H_- + 2n,~G_- + H^L + 2n|G^R+H_- + 2n,~G_- + H^R + 2n\}_- \\ = &\{G^L + H_-,~G_- + H^L|G^R+H_-,~G_- +H^R\}_- + 2n
\end{align*}
We conclude that for $n \gg 0$,
\[ G_- + H_- + 2n = \{G^L + H_-, G_- + H^L| G^R + H_-, G_- + H^R\}_- + 2n,\]
and so (\ref{ugh}) holds.
\end{proof}

Our tool for reducing $\mathcal{J}$ and $\mathcal{K}$ to the theory of partizan games will be the following maps:
\begin{definition}
Let $\mathbf{Pg}$ denote the class of short partizan games, modulo the usual equivalence.
Let $\psi, \psi^+,\psi^-$ be the three maps $\mathcal{W} \to \mathbf{Pg}$ defined recursively as follows:
\[ \psi(n) = \psi^+(n) = \psi^-(n) = n\]
for $n$ an integer, and otherwise,
\begin{align*} \psi(\{G^L|G^R\}) &= \{\psi(G^L)|\psi(G^R)\} \\
 \psi^+(\{G^L|G^R\}) &= \{\psi^+(G^L)|\psi^+(G^R)\}_+ \\
 \psi^-(\{G^L|G^R\}) &= \{\psi^-(G^L)|\psi^-(G^R)\}_- \end{align*}
Here the $\{\cdots|\cdots\}$ on the left refer to scoring games, while the $\{\cdots|\cdots\}$ on the right refer to
partizan games.
\end{definition}

So for instance, $\psi(\{-1|1\}) = 0$ because $\{-1|1\} = 0$ in the standard partizan theory.  Also note
that the scoring game $1$ has no options, but $\psi(1) = 1 = \{0|\}$ has a left option.

\begin{theorem}\label{plus-minus-homomorphism}
Let $G, H$ be well-tempered scoring games.  Then
\[ \psi^\pm(G + H) = \psi^\pm(G) + \psi^\pm(H)\]
\end{theorem}
\begin{proof}
We given an inductive proof for the $\psi^+$ case.  In the base case where $G, H$ are integers, this is trivial.
In the case where one of $G$ or $H$ is an integer and the other is not, this is Lemma~\ref{automatic-integer-avoidance}.
In the case where $G$ and $H$ are both integers, this is Lemma~\ref{hopefully-this-is-true}.
\end{proof}
Note that the analogous identity fails for $\psi$, since
\[ \psi(\{-2|2\} + 3) = \psi(\{1|5\}) = 2 \ne \{-2|2\} + 3 = \psi(\{-2|2\}) + \psi(3).\]
Even for $\psi^\pm$, these maps are discarding a lot of information, since $\psi^+(\{-2|2\}) = -1 = \psi^+(\{-2|0\})$, although
$\{-2|2\}$ and $\{-2|0\}$ are quite different as scoring games.

Nevertheless, $\psi$ preserves some information about outcomes.
\begin{lemma}\label{outcome-somewhat-preserved}
Let $G$ be a well-tempered scoring game.  If $G$ is even, then $\psi(G) \ge 0$ if and only if $\rout(G) \ge 0$.
If $G$ is odd, then $\psi(G) \rhd 0$ if and only if $\lout(G) \ge 0$.

(By symmetry, if $G$ is even, then $\psi(G) \le 0$ if and only if $\lout(G) \le 0$, and if $G$ is odd,
then $\psi(G) \lhd 0$ if and only if $\rout(G) \le 0$.)
\end{lemma}
\begin{proof}
We proceed by induction.  If $G$ is an integer $n$, then $\psi(G) = n$, which is $\ge 0$ if and only if it is $\ge 0$, as desired.

If $G$ is not an integer, but is even-tempered, then $\psi(G) \ge 0$ if and only if every $G^R$ satisfies $\psi(G^R) \rhd 0$.  By induction,
this means that $\lout(G^R) \ge 0$ for every $G^R$, which is the same as saying that $\rout(G) \ge 0$.

Finally, suppose that $G$ is odd-tempered.  Then $\psi(G) \rhd 0$ if and only if there exists a $\psi(G)^L = \psi(G^L) \ge 0$.
By induction, this is equivalent to $\rout(G^L) \ge 0$.  So $\psi(G) \rhd 0$ if and only if $\rout(G^L) \ge 0$ for some $G^L$.
But this is equivalent to $\lout(G) \ge 0$.
\end{proof}

The main place where the maps $\psi$ and $\psi^\pm$ are useful is on $\mathcal{J}$ and $\mathcal{K}$:
\begin{theorem}\label{reduction-1}
Let $G \in \mathcal{J}$. Then
\begin{description}
\item[(a)]
 $\psi(G) = \psi^+(G) = \psi^-(G)$
\item[(b)]
If $G$ is even-tempered and $n \in \mathbb{Z}$, then $n \ge \psi(G)$ if and only if $n \ge \lout(G)$
and $n \le \psi(G)$ if and only if $n \le \rout(G)$.
\item[(c)]
If $G$ is odd-tempered and $n \in \mathbb{Z}$, then $n \rhd \psi(G)$ if and only if $n \ge \rout(G)$
and $n \lhd \psi(G)$ if and only if $n \le \lout(G)$.
\end{description}
\end{theorem}
\begin{proof}
Let $|G|$ denote the total number of subgames of $G$.  We prove all three statements together, by induction on $|G|$.

In the base case where $|G| = 1$, $G$ is an integer, and all three statements are more or less trivial (the third one is vacuous).

Next suppose that we have proven all three statements for all games $G \in \mathcal{J}$ with $|G| < m$.  We first prove (a)
for all games $G \in \mathcal{J}$ with $|G| = m$.  Proving (a) amounts to showing
that there is at most one integer between the left and right options of
\[ \{\psi(G^L)|\psi(G^R)\}.\]
If not, then there is some $n$ such that
\begin{equation} \psi(G^L) \lhd n < n+1 \lhd \psi(G^R)\label{eq:problematic}\end{equation}
for all $G^L$ and $G^R$.  Because $|G^L| < m$ and $|G^R| < m$, we can apply (b) or (c) to these games by induction.
If $G$ is even-tempered, then by (c),
\[ \rout(G^L) \le n < n + 1 \le \lout(G^R)\]
for all $G^L$ and $G^R$.  Therefore,
\[ \lout(G) \le n < n + 1 \le \rout(G),\]
so $\lout(G) < \rout(G)$ and $\gap_0(G) > 0$, contradicting $G \in \mathcal{I}$.  Similarly, if $G$ is odd-tempered, then by (b) and (\ref{eq:problematic}),
\[ \rout(G^L) < n < n + 1 < \lout(G^R),\]
for all $G^L$ and $G^R$, so that $\gap_1(G) > 2$, contradicting $G \in \mathcal{J}$.  Thus (a) holds for all $G \in \mathcal{J}$
with $|G| = m$.

Next, we show that (b) holds for all even-tempered $G \in \mathcal{J}$ with $|G| = m$.  Note that $|G - n| = |G|$, so from (a),
\[ \psi(G - n) = \psi^+(G - n)\]
\[ \psi(G) = \psi^+(G)\]
Then by Theorem~\ref{plus-minus-homomorphism}, we get
\[ \psi(G - n) = \psi(G) - \psi^+(n) = \psi(G) - n\]
So we are reduced to showing that
\[ 0 \ge \psi(G - n) \iff 0 \ge \lout(G - n)\]
\[ 0 \le \psi(G - n) \iff 0 \le \rout(G - n)\]
So we are reduced to the case $n = 0$.  But this is just Lemma~\ref{outcome-somewhat-preserved}.  The proof of (c) is similar.
\end{proof}
\begin{corollary}\label{reduction-1.5}
Let $G, H \in \mathcal{J}$.  Then
\[ \psi(G + H) = \psi(G) + \psi(H)\]
\[ \psi(-G) = -\psi(G),\]
and when $\pi(G) = \pi(H)$,
\[ G \gtrsim H \iff \psi(G) \ge \psi(H)\]
\end{corollary}
\begin{proof}
The first claim follows from part (a) of the theorem together with Theorem~\ref{plus-minus-homomorphism}.  The second claim holds for any game $G$.
The third claim follows from the fact that $G - H$ is even, and so
\[ G \gtrsim H \iff \rout(G - H) \ge 0 \iff \psi(G - H) \ge 0 \iff \psi(G) - \psi(H) \ge 0 \iff \psi(G) \ge \psi(H)\]
using Theorem~\ref{i-games-nice-2}, Lemma~\ref{outcome-somewhat-preserved} (which is the $n = 0$ case of part (b) of Theorem~\ref{reduction-1}),
and the first two claims of this corollary.
\end{proof}
In particular, $\psi$ induces a well-defined map
\[ \mathbf{J} \rightarrow \mathbf{Pg},\]
and this map is an order-preserving homomorphism of groups.  Moreover, if we restrict to the even-tempered elements of $\mathbf{J}$, it is injective
and strictly order-preserving.  The kernel of the map $\mathbf{J} \to \mathbf{Pg}$ consists of the two elements $0$ and $\{-1|1\}$.

The analog of Theorem~\ref{reduction-1} for $\mathcal{K}$ is
\begin{theorem}\label{reduction-2}
Let $G \in \mathcal{K}$.  Then $\psi^\pm(G) = \psi(G^\pm)$.  In particular, if $G$ is even then
\[ n \ge \psi^+(G) \iff n \ge \lout(G^+) \iff n \gtrsim G^+ \iff n \gtrsim G\]
\[ n \ge \psi^-(G) \iff n \ge \lout(G^-) = \lout(G)\]
\[ n \le \psi^+(G) \iff n \le \rout(G^+) = \rout(G)\]
\[ n \le \psi^-(G) \iff n \le \rout(G^-) \iff n \lesssim G^- \iff n \lesssim G.\]
Similarly, if $G$ is odd then
\[ n \rhd \psi^-(G) \iff n \ge \rout(G^-) = \rout(G)\]
\[ n \lhd \psi^+(G) \iff n \le \lout(G^+) = \lout(G)\]
\end{theorem}
\begin{proof}
Only the claim that $\psi^+(G) = \psi(G^+)$ requires proof, since everything else follows by symmetry, Theorem~\ref{reduction-1}, or Theorem~\ref{structure}.
We proceed by  induction on $G$.  The case where $G \in \mathbb{Z}$ is clear.
Otherwise, let $G = \{G^L|G^R\}$.  Let $H^L$ and $H^R$ be upsides of the options of $G$, chosen to lie in $\mathcal{J}$, and let $H = \{H^L|H^R\}$.
Then $G^L \approx_+ H^L$ and $G^R \approx_+ H^R$ so by Lemma~\ref{obvious-compatibles} $G \approx_+ H$.  We consider the two cases $H \in \mathcal{I}$
and $H \notin \mathcal{I}$ separately.

First suppose that $H \in \mathcal{I}$.  Then $H \approx_+ G$ implies that $H \approx G^+ \in \mathcal{J}$, so $\lout(H) - \rout(H) \ge -2 \cdot \pi(G)$.
Since every option of $H$ is in $\mathcal{J}$, this shows that $H \in \mathcal{J}$.  Then by Theorem~\ref{reduction-1} and Corollary~\ref{reduction-1.5},
\[ \{\psi(H^L)|\psi(H^R)\}_+ = \{\psi^+(H^L)|\psi^+(H^R)\}_+ = \psi^+(H) = \psi(H) = \psi(G^+).\]
But by induction, $\psi(H^L) = \psi^+(G^L)$ and $\psi(H^R) = \psi^+(H^R)$.  So
\[ \psi^+(G) = \{\psi^+(G^L)|\psi^+(G^R)\}_+ = \{\psi^+(H^L)|\psi^+(H^R)\}_+ = \psi(G^+)\]
as desired.

Next, suppose that $H \notin \mathcal{I}$.  Every option of $H$ is in $\mathcal{I}$, so the only way that $H \notin \mathcal{I}$
can occur is if $H$ is even and $\lout(H) < \rout(H)$.  Then by Lemma~\ref{gap-rule},
\[ G \approx_+ H \approx_+ \rout(H).\]
Therefore $G^+ \approx \rout(H)$.  Let $n = \rout(H)$.  By induction and Theorem~\ref{reduction-1}(a),
\[ \psi^+(G) = \{\psi^+(G^L)|\psi^+(G^R)\}_+ = \{\psi(H^L)|\psi(H^R)\}_+.\]
To show that $\psi^+(G)$ equals $\psi(G^+) = n$, we need to show that $n$ is the maximum integer $x$ such that
\begin{equation*}
\psi(H^L) \lhd x \lhd \psi(H^R)
\end{equation*}
for all $H^L$ and $H^R$.  By Theorem~\ref{reduction-1}(c), this is equivalent to
\[ \rout(H^L) \le x \le \lout(H^R)\]
for all $H^L$ and $H^R$, or equivalently, $\lout(H) \le x \le \rout(H)$.  Clearly $n = \rout(H)$ is the greatest such $x$, using the fact
that $\rout(H) > \lout(H)$.
\end{proof}

\subsection{From Partizan Games to Scoring Games}
While $\mathcal{K}$ will be useful \S \ref{sec:boolean}, for the rest of this section we focus on $\mathcal{J}$.  We will find an inverse
to the homomorphism $\mathbf{J}^0 \to \mathbf{Pg}$, where $\mathbf{J}^0$ is the even-tempered subgroup of $\mathbf{J}$.

Let $\phi_0$ and $\phi_1$ be the maps from partizan games to well-tempered scoring games defined as follows:
\begin{align*} \phi_0(G) &= \begin{cases}
n & \text{ if $G$ equals (is equivalent to) an integer $n$} \\
\{\phi_1(G^L)|\phi_1(G^R)\} & \text{ otherwise}
\end{cases} \\
 \phi_1(G) &= \begin{cases}
\{n-1|n+1\} & \text{ if $G$ equals (is equivalent to) an integer $n$} \\
\{\phi_0(G^L)|\phi_0(G^R)\} & \text{ otherwise}
\end{cases}
\end{align*}
For example, $\phi_0(\{1|3*\}) = 2$, $\phi_0(\{1|-3*\}) = \{0|2||-3|-3\}$.  Note that $\phi_i(G)$ depend on the form of $G$, not just its value,
so we do not get well-defined maps $\mathbf{Pg} \to \mathcal{W}$.  Modulo equivalence, however, the $\phi_i$ maps are an exact inverse of $\psi$:
\begin{theorem}\label{reduction-3}
If $G$ is a partizan game and $i = 0$ or $1$, then $\phi_i(G)$ is a well-tempered game, with $\pi(\phi_i(G)) = i$.  Moreover,
$\phi_i(G) \in \mathcal{J}$ and $\psi(\phi_i(G)) = G$.
\end{theorem}
\begin{proof}
The fact that $\phi_i(G)$ is a well-tempered game of the correct parity follows by an easy induction, as does the fact that $\psi(\phi_i(G)) = G$.
We prove the remaining statement, that $\phi_i(G) \in \mathcal{J}$, by induction on $i$.
If $G$ is equal in value to an integer $n$, then $\phi_i(G)$ is either $n$ or $\{n-1|n+1\}$.  Either way, it is in $\mathcal{J}$.

Otherwise, $G$ does not equal an integer.  By induction, every option of $\phi_i(G)$ is in $\mathcal{J}$.  It remains to show that
\begin{equation}\lout(\phi_0(G)) \ge \rout(\phi_0(G))\label{ineq:erti}\end{equation}
and
\begin{equation} \lout(\phi_1(G)) \ge \rout(\phi_1(G)) - 2\label{ineq:ori}\end{equation}
I claim that these inequalities hold strictly.  Suppose not.  If (\ref{ineq:erti}) fails to hold strictly, then there exists some integer $n$
such that
\[ \lout(\phi_0(G)) \le n \le \rout(\phi_0(G))\]
Thus
\[ \rout(\phi_1(G^L)) \le n \le \lout(\phi_1(G^R))\]
for every left option $G^L$ of $G$ and every right option $G^R$ of $G$.  But by induction, every $\phi_1(G^L)$ and $\phi_1(G^R)$
is in $\mathcal{J}$, so Theorem~\ref{reduction-1} applies, giving
\[ G^L = \psi(\phi_1(G^L)) \lhd n \lhd \psi(\phi_1(G^R)) = G^R.\]
By the simplicity rule, $G$ equals an integer - the simplest $n$ between all the $G^L$ and $G^R$.
This contradicts the assumption that $G$ was not an integer, so (\ref{ineq:erti}) holds strictly.  The argument for (\ref{ineq:ori}) is similar.
%
\end{proof}
\begin{corollary}
The partially ordered groups $\mathbf{J}$ and $\mathbf{Pg} \times (\mathbb{Z}/2\mathbb{Z})$ are isomorphic.  The isomorphism
\[ \mathbf{J} \to \mathbf{Pg} \times (\mathbb{Z}/2\mathbb{Z}) \]
is induced by $G \mapsto (\psi(G),\pi(G))$, and its inverse is induced by $(H,i) \mapsto \phi_i(H)$.

Consequently, the $\phi_i$ maps are order-preserving homomorphisms, in the sense that
\[ \phi_i(G) \gtrsim \phi_i(H) \iff G \ge H\]
and
\[ \phi_{i+j}(G +H) \approx \phi_i(G) + \phi_j(H)\]
\end{corollary}
This follows easily from Theorem~\ref{reduction-3} and Corollary~\ref{reduction-1.5}.%

\subsection{Heating things up}

To complete the description of $\mathbf{I}^0$, we need to find the effect of heating on $\mathbf{Pg}$ under
the identification of $\mathbf{J}^0 \cong \mathbf{Pg}$.

If $K \ge 0$ is a partizan game, we can recursively define a function $f_K$ on partizan games by
\[ f_K(G) = \{f_K(G^L) + K | f_K(G^R) - K\}.\]
It is easy to see that this map is a homomorphism and is exactly order preserving: $f_K(G) \ge 0$ if and only if $G \ge 0$.
Thus $G = H$ implies that $f_K(G) = f_K(H)$.
In fact, this map is Norton multiplication by $\{K|\}$, or equivalently, the overheating operation $\int_{\{K|\}}^{K}$.
Because it is order-preserving, it is well-defined on the quotient space $\mathbf{Pg}$ of partizan games modulo equivalence.

Taking $K = 1*$, we get Norton multiplication by $\{1*|\}$, or equivalently, the overheating operation $G \mapsto \int_1^{1*} G$.
Note that $\int_1^{1*} n = n$ for $n \in \mathbb{Z}$ because $\int_1^{1*}1 = 1$.

As we noted above, $\mathcal{J}$ is closed under the operation $\int^t$ for $t > 0$, by (\ref{heat-gap}) of Lemma~\ref{early-lemma}.
\begin{theorem}\label{heating-effect}
Let $G$ be a well-tempered scoring game in $\mathcal{J}$. If $G$ is even-tempered, then
\[ \psi\left(\int^1 G\right) = \int_1^{1*} \psi(G)\]
and if $G$ is odd-tempered, then
\[ \psi\left(\int^1 G\right) = * + \int_1^{1*} \psi(G).\]
\end{theorem}
\begin{proof}
Since $\psi$ and $\phi$ are inverses of each other, it suffices to show that if $H$ is a partizan game, then
\begin{equation} \phi_0\left(\int_1^{1*} H\right) \approx \int^1 \phi_0(H) \label{phiside0} \end{equation}
\begin{equation} \phi_1\left(* + \int_1^{1*} H\right) \approx \int^1 \phi_1(H) \label{phiside1} \end{equation}
These two statements are equivalent, for fixed $H$, because
\[ \phi_1\left(* + \int_1^{1*} H\right) = \phi_1(*) + \phi_0\left(\int_1^{1*} H\right) = \{0|0\} + \phi_0\left(\int_1^{1*} H\right)\]
\[ \int^1 \phi_1(H) = \int^1 (\phi_1(0) + \phi_0(H)) = \int^1 (\{-1|1\} + \phi_0(H)) = \{0|0\} + \int^1 \phi_0(H).\]

We proceed to prove (\ref{phiside0}-\ref{phiside1}) by induction on $H$.
If $H$ is equal to an integer $n$, then $\int_1^{1*} H = \int_1^{1*} n = n$.  Therefore $\int_1^{1*} H$ also
equals an integer, and
\[ \phi_0\left(\int_1^{1*} H\right) = \phi_0(n) = n.\]
But $n = \int^1 n = \int^1 \phi_0(n) \approx \int^1 \phi_0(H)$, since $n = H$ implies that $\phi_0(n) \approx \phi_0(H)$.  This establishes (\ref{phiside0})
for $H$ equal to an integer.

When $H$ does not equal an integer, neither does
$\int_1^{1*} H$.  Then
\begin{align*} \phi_0\left(\int_1^{1*} H\right) &= \phi_0\left(\left\{\int_1^{1*} H^L + 1* | \int_1^{1*} H^R + 1*\right\}\right) \\
&= \left\{ \phi_1\left(\int_1^{1*} H^L + 1*\right) | \phi_1\left(\int_1^{1*} H^R + 1*\right)\right\} \\
&\approx \left\{ \phi_1\left(* + \int_1^{1*} H^L\right) + \phi_0(1) | \phi_1\left(* + \int_1^{1*} H^R\right) + \phi_0(1)\right\} \\
&\approx \left\{ \int^1 \phi_1(H^L) + 1 | \int^1 \phi_1(H^R) - 1 \right\} \\
&= \int^1 \{\phi_1(H^L)|\phi_1(H^R)\} \\
&= \int^1 \phi_0(H)
\end{align*}
proving (\ref{phiside0}).  Here, the first equality is the definition of $\int_1^{1*}$, the second follows
because $\int_1^{1*} H$ does not equal an integer, the third follows because $\phi$ is a homomorphism, the fourth follows
by induction and the fact that $\phi_0(1) = 1$, the fifth follows by definition of heating, and the sixth follows because $H$ does not equal
an integer.
%
%
\end{proof}

Consider the following families of well-tempered scoring games:
\begin{align*} \mathcal{F}_n &= \{G \in \mathcal{W} : \gap_0(G) = 0,~ \gap_1(G) \le 2n \}\\
               \mathcal{F}_n^0 &= \{G \in \mathcal{F}_n : \pi(G) = 0\} \end{align*}
Then we have
\[ \mathcal{J} = \mathcal{F}_{1} \subset \mathcal{F}_{2} \subset \mathcal{F}_3 \subset \cdots \subset \mathcal{I}\]
and in fact $\mathcal{I} = \bigcup_{i = 1}^n \mathcal{F}_i$.  By (\ref{heat-gap}) of Lemma~\ref{early-lemma}, $\int^j$ is an isomorphism
from $\mathcal{F}_{n+j} \to \mathcal{F}_n$ for every $n, j$.  Also,
by Theorem~\ref{gap-sums}, each $\mathcal{F}_n$ is closed under addition.  Since they are also clearly closed under negation,
it follows that the $\{\mathbf{F}_n\}$ form a filtration of subgroups of $\mathbf{I}$.

Now $\mathbf{I}^0$ is the direct limit
of the following sequence of partially-ordered abelian groups:
\[
\mathbf{F}^0_1 \hookrightarrow \mathbf{F}^0_2 \hookrightarrow \mathbf{F}^0_3 \hookrightarrow \cdots 
\]
where the maps are inclusions.  But this is the top row of the following commutative diagram, whose vertical maps are isomorphisms:
\[
\begin{CD}
\mathbf{F}^0_1 @>>> \mathbf{F}^0_2 @>>> \mathbf{F}^0_3 @>>> \cdots \\
@V{id}VV         @V{\int^1}VV           @V{\int^2}VV   \\
\mathbf{J}^0 @>{\int^1}>> \mathbf{J}^0 @>{\int^1}>> \mathbf{J}^0 @>{\int^1}>> \cdots \\
@V{\psi}VV         @V{\psi}VV            @V{\psi}VV \\
\mathbf{Pg} @>{\int_1^{1*}}>> \mathbf{Pg} @>{\int_1^{1*}}>> \mathbf{Pg} @>{\int_1^{1*}}>>
\end{CD}
\]
Consequently, as an abstract partially-ordered abelian group, $\mathbf{I}^0$ is isomorphic to the direct limit
\[ \mathbf{Pg} \stackrel{\int_1^{1*}}{\rightarrow} \mathbf{Pg} \stackrel{\int_1^{1*}}{\rightarrow} \mathbf{Pg} \stackrel{\int_1^{1*}}{\rightarrow} \cdots \]
We can do something similar with $\mathbf{I}$, expressing it as a direct limit of copies of $\mathbf{Pg} \times (\mathbb{Z}/2\mathbb{Z})$.

Combined with Theorems~\ref{structure} and \ref{structure-small}, this gives us a complete description of $\mathcal{W}$ in terms of
This in turn gives us a complete description of $\mathbf{W}$ in terms of $\mathbf{Pg}$.  To recover right and left outcomes from this description,
one can utilize the following facts (all of which have been previously proven):
\begin{itemize}
\item If $G$ is even-tempered, then $\lout\left(\int^t G\right) = \lout(G)$.
\item If $G$ is odd-tempered, then $\lout\left(\int^t G \right) = \lout(G) + t$.
\item If $G \in \mathcal{J}$ is even-tempered, then
\[ \lout(G) = \min\{n \in \mathbb{Z} : n \ge \psi(G)\}\]
\item If $G \in \mathcal{J}$ is odd-tempered, then
\[ \lout(G) = \max\{n \in \mathbb{Z} : n \lhd \psi(G)\}\]
\item If $G$ is even-tempered, then $\lout(G) = \rfout(G) = \rfout(G^-) = \lout(G^-)$.
\item If $G$ is odd-tempered, then $\lout(G) = \lfout(G) = \lfout(G^+) = \lout(G^+)$.
\item $\rout(G) = -\lout(-G)$.
\end{itemize}

\section{Canonical Forms for Invertible Games}\label{sec:canonical}
In the standard partizan theory, every equivalence class of games has a canonical representative, characterized by its lack of reversible and dominated moves.
In this section, we prove the same thing for well-tempered scoring games that are \emph{invertible}.  This section is included mainly for completeness,
and follows easily from the previous section, so the reader may want to skip it.

We first observe that non-invertible well-tempered scoring games need not have any sort of canonical form.
For instance,
\[ \{x|0||1|y\} \approx 1\&0\]
for any numbers $x, y \in \mathbb{Z}$, and there is no sense in which any of these forms is minimal or canonical.  Also, no simpler game
is a form of $1\&0$.  Probably the most canonical way of describing a general scoring game $G$ is as
\[ G^+ \& G^-,\]
using the canonical forms of $G^+$ and $G^-$ which we will describe below.

\begin{definition}
Let $G$ be a well-tempered scoring game.  Then we say that a left option $G^L$ is \emph{reversible} if there is some right option $G^{LR}$
of $G^L$ with $G^{LR} \lesssim G$.  We say that a right option $G^R$ is \emph{reversible} if there is some left option $G^{RL}$ of $G^R$
with $G^{RL} \gtrsim G$.  We say that a left option $G^L$ is \emph{dominated} if some other left option $G^{L'}$ satisfies
$G^{L'} \gtrsim G^L$, and we say that a right option $G^R$ is \emph{dominated} if some other right option $G^{R'}$ satisfies
$G^{R'} \lesssim G^R$.  If $G \in \mathcal{I}$, we say that $G$ is \emph{canonical} if it has no reversible or dominated options, and the
same is true for all subgames.
\end{definition}
The rest of this section is a proof that every invertible $\mathbb{Z}$-valued game has a unique form which is canonical.
We reduce this to the same fact for partizan games.  There are probably simpler proofs.

\begin{lemma}\label{canonical-1}
If $G'$ is an option of $G$, then the corresponding option $\int^t G' \pm t$ of $\int^t G$ is dominated (resp. reversible)
if and only if $G'$ is dominated (resp. reversible).  Consequently, $G$ has dominated (resp. reversible) options if and only if $\int^t G$ does.
In particular, if $G \in \mathcal{I}$, then $\int^t G$ is canonical if and only if $G$ is.
\end{lemma}
We leave the easy proof as an exercise to the reader.

\begin{lemma}\label{canonical-2}
Let $G$ be a partizan game in canonical form.  Then $\phi_0(G)$ and $\phi_1(G)$ are canonical.
\end{lemma}
\begin{proof}
This follows easily from induction, modulo the following fact: if $G$ is not an integer (in value), but has the integer $n$ as an option, then the option
$\phi_1(n) = \{n-1|n+1\}$ of $\phi_0(G)$ is not reversible.  Without loss of generality, $n$ is a left option of $G$.  Suppose
$\phi_1(n)$ is reversible.  Then $n + 1 \lesssim \phi_0(G)$.  Thus $n + 1 \le G$.  Then $n + 1 \lhd G^R$ for every right option of $G$. Since $G$ is
not an integer, $G^L \lhd n + 1$ cannot hold for all $G^L$.  So for some $G^L$,
\[ n < n + 1 \le G^L.\]
Then $n$ is a dominated left option of $G$, a contradiction.
%
\end{proof}

\begin{lemma}\label{canonical-3}
Let $G$ be a well-tempered scoring game for which $\gap_0(G) = 0$ and $\gap_1(G) \le 1$.  If $G$ is canonical, then $\psi(G)$ is in canonical form
(assuming that $\psi(n)$ was chosen to be in canonical form for all integers $n$).
\end{lemma}
\begin{proof}
This theorem follows trivially by induction, except that we need to verify one special case: if $G$ is odd-tempered and $n \in \mathbb{Z}$
is an option of $G$, then $\psi(n)$ is not a reversible option of $\psi(G)$.  Without loss of generality, $n$ is a left option of $G$.
We need to show that no right option of $\psi(n) = n$ reverses Left's move from $\psi(G)$ to $\psi(n)$.  If $n \ge 0$, then $n$ has no right options,
so suppose $n < 0$.  Since we are taking $\psi(n)$ in canonical form, the only right option of $n$ is $n + 1$.  Thus, we are reduced to proving that
if $G$ is odd-tempered and $\gap_1(G) \le 1$ and $n$ is a left option of $G$, then $n+1 \not \le \psi(G)$.  By Theorem~\ref{reduction-1},
this amounts to showing that $n + 1 \ge \rout(G)$.  But otherwise, we would have
\[  \lout(G) =  n < n + 1 < \rout(G) \]
contradicting $\gap_1(G) \le 1$.
\end{proof}

\begin{lemma}\label{canonical-4}
If $G \in \mathcal{I}$ is canonical, and $G \approx n$ for some integer $n$, then $G = n$.
\end{lemma}
\begin{proof}
Suppose for the sake of contradiction that $G$ is not an integer.  Then $\rout(G) = n$, so there is some
right option $G^R$ of $G$ with $\lout(G^R) = n$.  Since $G^R$ is odd-tempered, it is not an integer, so there is some left option
$G^{RL}$ of $G^R$ with $\rout(G^{RL}) = n$.  As $G^{RL} - G \approx G^{RL} - n$, $\rout(G^{RL} - G) = \rout(G^{RL}) - n = 0$.  Then by
Theorem~\ref{i-games-nice}, $G^{RL} \gtrsim G$, and $G$ has a reversible right option, a contradiction.
\end{proof}

\begin{lemma}\label{canonical-5}
Let $G$ be a canonical well-tempered scoring game with $\gap_0(G) = 0$ and $\gap_1(G) \le 1$.  Then $\phi_{\pi(G)}(\psi(G)) = G$
(exactly, not up to equivalence).
\end{lemma}
\begin{proof}
We proceed by induction on $G$.  If $G$ is an integer $n$, then $\psi(G) \equiv n$, so $\phi_{\pi(G)}(\psi(G))$ is automatically
$n = G$.  Otherwise, if $G$ is even-tempered, then by induction
\[ \phi_0(\psi(G)) = \{\phi_1(\psi(G)^L)|\phi_1(\psi(G)^R)\} = \{\phi_1(\psi(G^L))|\phi_1(\psi(G^R))\} = \{G^L|G^R\}\]
\emph{unless} $\psi(G)$ equals an integer.  But if $\psi(G) = m$ for some $m \in \mathbb{Z}$,
then $G \approx \phi_0(m) = m$.  Then by Lemma~\ref{canonical-4}, $G = m$, a contradiction.

Finally suppose $G$ is odd-tempered.  Then again, by induction,
\[ \phi_1(\psi(G)) = \{\phi_0(\psi(G^L))|\phi_0(\psi(G^R))\} = \{G^L|G^R\}\]
\emph{unless} $\psi(G)$ equals an integer.  But if $\psi(G) = m$, then $G \approx \phi_1(m) = \{m-1|m+1\}$, so
\[ \rout(G) = \rout(\{m-1|m+1\}) = m+1\]
\[ \lout(G) = \lout(\{m-1|m+1\}) = m - 1\]
and thus $\gap_1(G) \ge 2$, a contradiction.
\end{proof}

\begin{theorem}
If $G$ is invertible (modulo equivalence), then there is a unique $H \in \mathcal{I}$ with $H \approx G$ and $H$ canonical.
\end{theorem}
\begin{proof}
By Corollary~\ref{i-stands-for-invertible}, we can assume without loss of generality that $G \in \mathcal{I}$.  By taking $t$ sufficiently large, we get
$\int^t G \in \mathcal{J}$.  Let $x$ be the canonical form of $\psi\left(\int^t G\right)$.  Then by Lemma~\ref{canonical-2}, $K = \phi_{\pi(G)}(x)$ is canonical.
But since $x = \psi\left(\int^t G\right)$,
\[ K = \phi_{\pi(G)}(x) \approx \int^t G.\]
Then by Lemma~\ref{canonical-1}, $\int^{-t} K$ is canonical.  But it is in $\mathcal{I}$ and $\int^{-t} K \approx G$, so we have shown existence.

For uniqueness, suppose that $H$ and $H'$ are two canonical games in $\mathcal{I}$, with $H \approx H'$ but $H \ne H'$.  Then by taking $t$ sufficiently large, we can
arrange for
\[ \gap_1\left(\int^t H\right) \le 1 \ge \gap_1\left(\int^t H'\right)\]
Also, since $\int^t$ is invertible, $\int^t H \ne \int^t H'$ and $\int^t H \approx \int^t H'$.  So without loss of generality, $H$ and $H'$
have $\gap_1 \le 1$.  Then $\psi(H)$ and $\psi(H')$ are both in canonical form by Lemma~\ref{canonical-3}.  Since they are canonical forms
of equivalent games, $\psi(H) \equiv \psi(H')$.
Therefore, letting $i = \pi(H) = \pi(H')$
\[ \phi_i(\psi(H)) = \phi_i(\psi(H')).\]
But by Lemma~\ref{canonical-5}, the two sides equal $H$ and $H'$.
\end{proof}

\section{Boolean-valued Games}\label{sec:boolean}
In this section we focus on $\{0,1\}$-valued scoring games, which we call \emph{Boolean-valued games}.
There are two interesting order-preserving
functions $\{0,1\} \times \{0,1\} \to \{0,1\}$, namely logical \textsc{or} and logical \textsc{and}.  We denote these operations, and their
extensions to Boolean-valued games, by $\vee$ and $\wedge$, respectively.
Note that for $i,j \in \{0,1\}$, $i \vee j = \max(i, j)$ and $i \wedge j = \min(i, j)$.

One can view the two possible scores $0$ and $1$ as victory for Right and victory for Left, respectively.
In this interpretation, $G \wedge H$ and $G \vee H$ are compound games played as follows: $G$ and $H$ are played in parallel, like a disjunctive sum.
If one player wins both $G$ and $H$, she wins the compound.  If there is a tie, $G \vee H$ resolves it in favor of Left, and $G \wedge H$ resolves
it in favor of Right.  Thus $G \vee H$ and $G \wedge H$ can be seen as biased disjunctive sums.

By Theorem~\ref{structure-small}, $\vee$ and $\wedge$ induce well-defined order-preserving operations
\[ \mathbf{W}_{\{0,1\}} \times \mathbf{W}_{\{0,1\}} \to \mathbf{W}_{\{0,1\}}\]
which are characterized by their restrictions
\[ \mathbf{I}^0_{\{0,1\}} \times \mathbf{I}^0_{\{0,1\}} \to \mathbf{I}^0_{\{0,1\}}\]
to even-tempered invertible games.

Our goal is to determine the structure of $\mathbf{I}^0_{\{0,1\}}$ and the effects of $\wedge$ and $\vee$ on $\mathbf{I}^0_{\{0,1\}}$.
The main result in this direction is Theorem~\ref{boolean-summary}.

\subsection{Classification of Boolean-valued Games}
By Theorem~\ref{k-non-void}(a) and (d), every Boolean-valued game lives in $\mathcal{K}$, and every element of $\mathcal{I}_{\{0,1\}}$ is in $\mathcal{J}$.
Therefore, it is reasonable to apply the maps $\psi^\pm$ and $\psi$ to Boolean-valued games, and by Theorem~\ref{reduction-2}, $\psi^\pm(G) = \psi(G^\pm)$.

\begin{lemma}\label{boolean-psi-defaults}
If $G \in \mathcal{W}_{0,1}$, then $\psi(G) = \psi^-(G)$.
\end{lemma}
\begin{proof}
Let $G$ be a minimal counterexample.  Then $G$ is not an integer and
\[ \psi(G) = \{\psi(G^L)|\psi(G^R)\} \ne \psi^-(G) = \{\psi(G^L)|\psi(G^R)\}_-.\]
This can only happen if some negative integer $n$ satisfies $\psi(G^L) \lhd n \lhd \psi(G^R)$ for all $G^L$ and $G^R$.
In this case, $\psi^-(G)$ will be a negative integer $n < 0$.  But this means that $G^- \approx n$ or $G^- \approx \{n-1|n+1\}$, which are impossible
since $\lfout(G^-) \in \{0,1\}$ while the left outcomes of $n$ and $\{n-1|n+1\}$ are negative.
\end{proof}
It follows that
\begin{equation} \psi(\mathcal{I}_{\{0,1\}}) = \psi^+(\mathcal{W}_{\{0,1\}}) =  \psi^-(\mathcal{W}_{\{0,1\}}) =
 \psi(\mathcal{W}_{\{0,1\}}).\label{mystery-set}\end{equation}
The first two equalities hold on general principles, by Theorems~\ref{structure-small} and \ref{reduction-2}.

As we noted in \S \ref{sec:psi}, the overheating operation $\int_{\{1/2|\}}^{1/2} = \int_1^{1/2}$ defined recursively
on partizan games by
\[ \int_1^{1/2} G = \left\{\int_1^{1/2} G^L + \frac{1}{2} | \int_1^{1/2} G^R - \frac{1}{2} \right\}\]
is a well-defined order-preserving homomorphism.
Let $S$ be the following set of partizan games:
\[ S = \{0,\frac{1}{4},\frac{3}{8},\frac{1}{2},\frac{1}{2} + *, \frac{5}{8},\frac{3}{4}, 1\}.\]
\begin{theorem}\label{boolean-theorem}
The set $\psi(\mathcal{I}_{\{0,1\}}) = \psi(\mathcal{W}_{\{0,1\}})$ of Equation (\ref{mystery-set}) is the following:
\[ \left\{\int_1^{1/2} x : x \in S\right\} \cup \left\{* + \int_1^{1/2} x : x \in S\right\}\]
More specifically,
\begin{equation} \{\psi(G) : G \in \mathcal{W}_{\{0,1\}},~\pi(G) = 0\} = \left\{\int_1^{1/2} x : x \in S\right\} \label{boolean-even}\end{equation}
and
\begin{equation} \{\psi(G) : G \in \mathcal{W}_{\{0,1\}},~\pi(G) = 1\} = \left\{* + \int_1^{1/2} x : x \in S\right\} \label{boolean-odd}\end{equation}
Also, the two sets $\left\{\int_1^{1/2} x : x \in S\right\}$ and $\left\{* + \int_1^{1/2} x : x \in S\right\}$ are disjoint.
\end{theorem}
\begin{proof}
We omit most of the details of the proof, since it consists entirely of calculations.  One can directly check
that if $x \in S$, then
\[ \phi_0\left(\int_1^{1/2} x\right) \text{ and } \phi_1\left(* + \int_1^{1/2} x\right)\]
are $\{0,1\}$-valued.  So all the specified values occur as $\psi(G)$ for various Boolean-valued games $G$.  The fact
that the two sets $\left\{\int_1^{1/2} x : x \in S\right\}$ and $\left\{* + \int_1^{1/2} x : x \in S\right\}$ can be checked by examination.
The only thing left is a proof that \emph{only} the specified values of $\psi$ occur as values of Boolean-valued games.
Letting $X_0 = \left\{\int_1^{1/2} x : x \in S\right\}$ and $X_1 = \left\{* + \int_1^{1/2} x : x \in S\right\}$, we assert the following:
\begin{claim}\label{boolean-claim}
If $G$ is a partizan game with more than zero left options and more than zero right options, and if every option of $G$ is in $X_i$,
then $G$ is (equivalent to a game) in $X_{1-i}$.
\end{claim}
Since the sets $X_i$ are finite, this can be checked directly.  We might as well assume that the left options of $G$ are pairwise incomparable,
and the same for the right options.  This speeds up the verification, since the posets $X_0$ and $X_1$ have very few antichains, as they are almost linearly ordered.
This leaves less than two hundred cases to check.  Another useful trick is to prove that for any even-tempered game $G$ with $\gap_0(G) = 0$
and $\gap_1(G) \le 1$,
\[ \psi\left(\int^{1/2} G\right) \approx \int_1^{1/2} \psi(G).\]
Then, one can cool all the Boolean-valued games by $1/2$ (which leaves them in $\mathcal{J}$), and one is reduced to considering numbers.
The simplicity rule now comes into play, and the verification becomes even easier.  There are some technical difficulties involved with applying
the $\int^{-1/2}$ operator to odd-tempered games, since this yields half-integer-valued scoring games.
Nevertheless, this approach works, but still degenerates into a case-by-case exhaustion of the possibilities, so we do not pursue it here.

Give Claim~\ref{boolean-claim}, it is easy to verify inductively that $\psi(G) \in X_{\pi(G)}$ for any Boolean-valued game $G$, which is the
remaining statement of the theorem.
\end{proof}

It follows from this theorem that $\mathbf{I}_{\{0,1\}}$ is isomorphic as a poset to the direct product of $S$ and $\mathbb{Z}/2\mathbb{Z}$,
since $\int_1^{1/2}$ is order-preserving. 
By directly considering the structure of the poset, we see that the number of pairs $(x,y) \in S \times S$ with $x \ge y$ is 35.
Therefore there are exactly $35$ even-tempered elements of $\mathbf{W}_{\{0,1\}}$ and exactly 70 elements total.  In other words,
there are 70 Boolean-valued games, up to equivalence.

\begin{definition}
Let $G$ be a Boolean-valued well-tempered game.  Then we define $u^+(G)$ and $u^-(G)$ to be the unique elements of $S$ such that
\[ \psi^\pm(G) = \int_1^{1/2} u^\pm(G)\]
if $G$ is even, and
\[ \psi^\pm(G) = * + \int_1^{1/2} u^\pm(G)\]
if $G$ is odd.
If $u^+(G) = u^-(G)$, we denote the common value by $u(G)$.  We call $u^+(G)$ and $u^-(G)$ \emph{the $u$-values of the game $G$}.
\end{definition}
Note that since $\psi^\pm(* + G) = * + \psi^\pm(G)$, we have $u^\pm(\even(G)) = u^\pm(G)$.

\subsection{Operations on Boolean-valued Games}
To finish our account of Boolean-valued well-tempered games, we need to describe how $u$-values interact with the operations $\wedge$ and $\vee$.
By Theorem~\ref{structure-small}, $\wedge$ and $\vee$ are determined by what they do to $\mathbf{I}^0_{\{0,1\}}$ which we can identify with $S$.
Therefore the following definition makes sense:
\begin{definition}
Let $\cup$ and $\cap$ be the binary operations on $S$ such that for any $G, H \in \mathcal{W}_{\{0,1\}}$,
\[ u^+(G \vee H) = u^+(G) \cup u^+(H)\]
\[ u^-(G \vee H) = u^-(G) \cup u^-(H)\]
\[ u^+(G \wedge H) = u^+(G) \cap u^+(H)\]
\[ u^-(G \wedge H) = u^-(G) \cap u^-(H)\]
\end{definition}
The fact that these identities hold for odd-tempered $G$ or $H$ follows from the equivalence $u^\pm(\even(G)) = u^\pm(G)$.
Note that $\cup$ and $\cap$ inherit commutativity and associativity from $\vee$ and $\wedge$.

\begin{theorem}\label{cup-cap-rule}
For any $x, y \in S$, $x \cup y$ is the maximum $z \in S$ such that $z \le x + y$.
Similarly, $x \cap y$ is the minimum element $w \in S$ such that $w \ge x + y - 1$.
\end{theorem}
Note that it is not a priori clear that such a $z$ or $w$ exist in general.
\begin{proof}
By symmetry(!) we only need to prove that $x \cup y$ has the stated form.
For $s \in S$, let $G_s \in \mathcal{I}_{(0,1)}$ be even-tempered with $u^\pm(G_s) = s$.  For $x, y, z \in S$, we need
to show that $z \le x \cup y \iff z \le x + y$, or equivalently,
$G_z \lesssim (G_x \vee G_y) \iff z \le x + y$.  Since $\int_1^{1/2}$ and $\psi$ are strictly order-preserving homomorphisms, this
is the same as showing that
\begin{equation} G_z \lesssim (G_x \vee G_y) \iff G_z \le G_x + G_y.\label{cup-cap-goal}\end{equation}
Let $f(n) = \min(n,1)$.  Then $f$ is the identity on $\{0,1\}$, $f(n) \le n$, and $\tilde{f}(G_x + G_y) = G_x \vee G_y$.
Thus
\[ \tilde{f}(G_z) = G_z \text{ and } (G_x \vee G_y) \lesssim G_x + G_y.\]
Therefore
\[ G_z \lesssim G_x + G_y \implies \tilde{f}(G_z) \lesssim (G_x \vee G_y) \implies G_z \lesssim (G_x \vee G_y) \implies G_z \lesssim G_x + G_y.\]
This proves (\ref{cup-cap-goal}).
\end{proof}

We summarize all our results on Boolean-valued well-tempered games in the following theorem:
\begin{theorem}\label{boolean-summary}
Let $S$ be the following set of partizan games:
\[ S = \{0, \frac{1}{4},\frac{3}{8}, \frac{1}{2}, \frac{1}{2}*, \frac{5}{8}, \frac{3}{4}, 1\}\]
For every $G \in \mathcal{W}_{\{0,1\}}$ we have elements $u^\pm(G) \in S$ such that
\[ \psi^\pm(G) = \int_1^{1/2} u^\pm(G)\]
if $G$ is even-tempered, and
\[ \psi^\pm(G) = * + \int_1^{1/2} u^\pm(G)\]
if $G$ is odd-tempered, where $\psi$ is the map of \S\ref{sec:psi}.

Moreover $\mathbf{W}_{\{0,1\}}$ is isomorphic to $\{(x,y,i) : x,y \in S,~x \ge y, i \in \{0,1\}\}$ via the map induced by
$G \mapsto (u^+(G), u^-(G), \pi(G))$.  This map is order-preserving, in the sense that
if $G$ and $H$ are two elements of $\mathcal{W}_{\{0,1\}}$, then $G \lesssim H$ if and only if
\[ u^+(G) \le u^+(H),~ u^-(G) \le u^-(H),~ \text{and}~ \pi(G) \le \pi(H).\]

If $G, H \in \mathcal{W}_{\{0,1\}}$ then
\[ u^\pm(G \wedge H) = u^\pm(G) \cap u^\pm(H)\]
\[ u^\pm(G \vee H) = u^\pm(G) \cup u^\pm(H)\]
where $x \cup y$ is the greatest element of $S$ less than or equal to $x + y$, and $x \cap y$ is the minimum element
of $S$ greater than or equal to $x+ y - 1$.

If $G \in \mathcal{W}_{\{0,1\}}$ is even-tempered, then
\[ \lout(G) = 0 \iff u^-(G) = 0\]
\[\rout(G) = 1 \iff u^+(G) = 1\]
while if $G$ is odd-tempered, then
\[ \lout(G) = 0 \iff u^+(G) \le \frac{1}{2} \]\[ \rout(G) = 1 \iff u^-(G) \ge \frac{1}{2}.\]
\end{theorem}
\begin{proof}
All of this follows by piecing together previous results of this section, except for the final statement about outcomes.
This can be proven from Theorem~\ref{reduction-2} and some facts about overheating, or by directly checking the sixteen invertible games.
(In fact, by monotonicity one only needs to check it for the even-tempered games with $u$-values $\frac{1}{4}$ and $\frac{3}{4}$, and the odd-tempered
games with $u$-values $\frac{1}{2}$ and $\frac{1}{2}*$.)
\end{proof}

\subsection{The Number of S-valued Games}
Let $S$ be a finite set of integers.  By Corollary~\ref{same-size}, the number of $S$-valued games, up to equivalence, depends only on $|S|$.
If $S = \emptyset$, then there are no $S$-valued games.  If $S = \{0\}$, it is easy to see that the only games are $0$ and $\{0|0\}$
(for example, use Lemma~\ref{i-games-nice} and the fact that any $S$-valued game must be in $\mathcal{I}$).  Above we saw that the number
of $\{0,1\}$-valued games, up to equivalence, is exactly seventy.  For the remaining case, we have the following theorem:

\begin{theorem}
If $S \subseteq \mathbb{Z}$ has more than two elements, then $\mathbf{I}^0_S$ contains a subposet isomorphic to the all-small games.
In particular, there are infinitely many well-tempered $S$-valued games.
\end{theorem}
\begin{proof}
We might as well take $S = \{-1,0,1\}$.  Let $G$ be any all-small game.  Then $\phi_0(G)$ and $\phi_1(G)$ are in $\mathcal{W}_{\{-1,0,1\}}$ by an easy-induction.
Since $\phi_0$ is faithful, having $\psi$ as its inverse, there is a complete copy of the partially ordered abelian group of all-small games
in $\mathbf{I}^0_{\{-1,0,1\}}$.
\end{proof}
In summary, we see that the size of $\mathbf{W}_S$ is given as follows:
\[ |\mathbf{W}_S| = \begin{cases}
0 & |S| = 0 \\
2 & |S| = 1 \\
70 & |S| = 2 \\
\infty & |S| \ge 3
\end{cases} \]
Since the group of all-small games contains a copy of $\mathbf{Pg}$ (via the embedding $G \mapsto G.\uparrow = \hat{G}$), it follows that the theory
of $S$-valued games is as complicated as the theory of short partizan games, for $|S| \ge 3$.  We can get at the additive structure of
the group of all-small games as follows:
\begin{definition}
Let $\star$ be the operation $\{-1,0,1\}^2 \to \{-1,0,1\}$ given by
\[ i \star j = \max(-1,\min(1,i+j)),\] as well as the extension of this
operation to $\{-1,0,1\}$-valued games.
\end{definition}
\begin{theorem}
If $G, H$ are all-small (partizan) games, then
\[ \phi_0(G+ H) \approx \phi_0(G) \star \phi_0(H)\]
\end{theorem}
We leave the proof as an exercise to the reader - it is similar to the proof of Theorem~\ref{cup-cap-rule}.

%

\section{Examples}\label{sec:Examples}
Very few naturally occurring scoring games are well-tempered.  The games \textsc{Triangles and Squares}, \textsc{Mercenary Clobber}, \textsc{Celestial Clobber},
 and \textsc{Rebel Hex} were invented by the author to give examples of scoring games.  The only pre-existing scoring games considered here
are \textsc{Hex}, \textsc{Misere Hex}, and \textsc{To Knot or Not to Knot}.

\subsection{Triangles and Squares}
The joke game of \textsc{Brussel Sprouts}
is a variant of Conway and Paterson's \textsc{Sprouts}, played as follows: initially the board is full of small
crosses.  Players take turns connecting the loose ends of the crosses.  Every time you connect two lose ends, you add another cross in the middle of your new
segment.  Play continues until someone is unable to move; this person loses.  A sample game is shown in Figure~\ref{brussel-sprouts-game}.

Every move in Brussel Sprouts preserves the total number of loose ends, but decreases
\[ \text{ the number of connected components} - \text{ the number of faces }\]
by one.  At the end of the game, each face contains exactly one loose end, and there is one connected component.
As there are initially $4n$ loose ends,
$n$ connected components, and one face, the total length of the game is $5n-2$ moves.  The game is played by the normal play rule,
so if $n$ is odd, the first player wins, and if $n$ is even, the second player wins, \emph{regardless of how the players play}.  In particular, there is
no strategy involved - this is the joke of Brussel Sprouts.  This makes Brussel Sprouts frivolous as a normal-play partizan game, but well-suited for our purposes.
Further information on Sprouts and Brussel Sprouts can be found in Chapter 17 of \cite{WinWays}.

\begin{figure}
\centering
\def \svgwidth{4in}
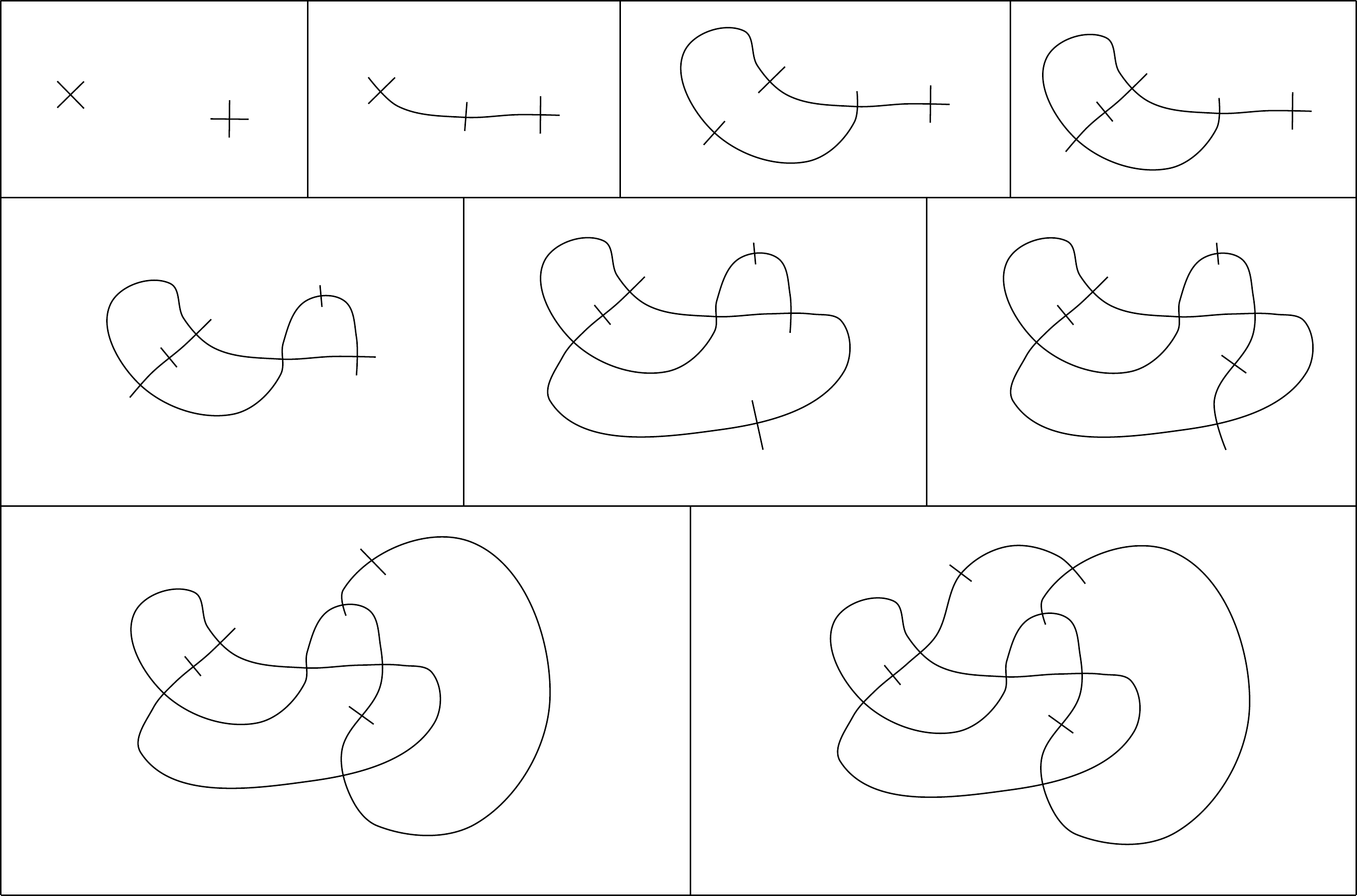 
\caption{A sample game of \textsc{Brussel Sprouts} (or of \textsc{Triangles and Squares}).  The initial position is in the top left, and
the successive positions are shown from left to right, top to bottom.  Because there were $n = 2$ crosses initially, the game lasts
$5n - 2 = 8$ turns.}
\label{brussel-sprouts-game}
\end{figure}

We define \textsc{Triangles and Squares} to be played exactly the same as \textsc{Brussel Sprouts} except that at the end, a score is assigned as follows:
Left gets one point for every ``triangLe'' (3-sided region), and loses one point for every ``squaRe'' (4-sided region).  Here, we say that a region
has $n$ sides if there are $n$ corners around its perimeter, not counting the two by the unique loose end.  We include the exterior unbounded
region in our count, as if the game were played on a sphere.  An example endgame is shown
in Figure~\ref{sample-side-count}.  Note that we have introduced an asymmetry
between Left and Right, so the game is no longer impartial.

\begin{figure}
\centering
\def \svgwidth{2in}
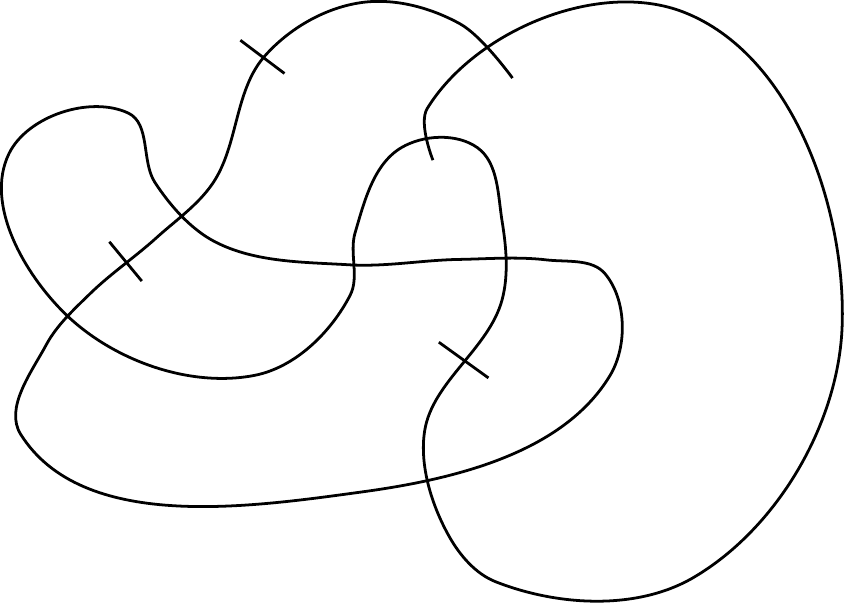 
\caption{The number of ``sides'' of each region in the final position of Figure~\ref{brussel-sprouts-game}.  Since there are two ``triangles''
and three ``squares,'' the final score is $2 - 3$ points for Left, or $3 - 2$ points for Right.}
\label{sample-side-count}
\end{figure}

Positions in Triangles and Squares naturally decompose as sums of smaller positions, making the
game amenable to our analysis.  Indeed, each cell (face) of a position acts independently of the others.
Figure~\ref{triangles-squares-dictionary} shows a few small positions and their values.  Using these values
of individuals cells, we can calculate the value of a compound position, as demonstrated in Figure~\ref{sample-calculation-t-s}.
Rather complicated
values seem to occur in this game, such as the position in Figure~\ref{octagon-t-s}.

\begin{figure}
\centering
\def \svgwidth{4in}
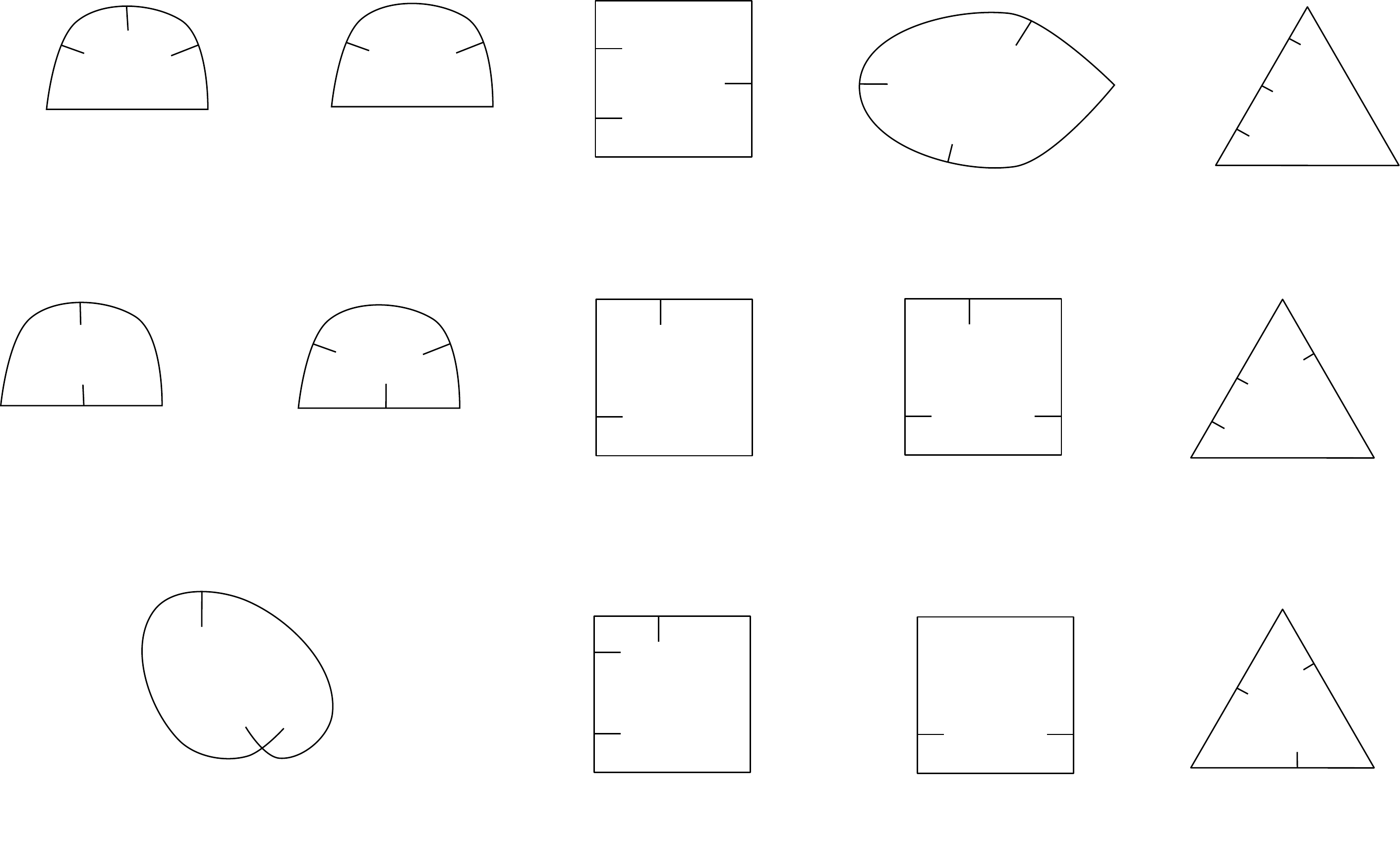 
\caption{The values of some small positions of \textsc{Triangles and Squares}, from Left's point of view.}
\label{triangles-squares-dictionary}
\end{figure}

\begin{figure}
\centering
\def \svgwidth{1in}
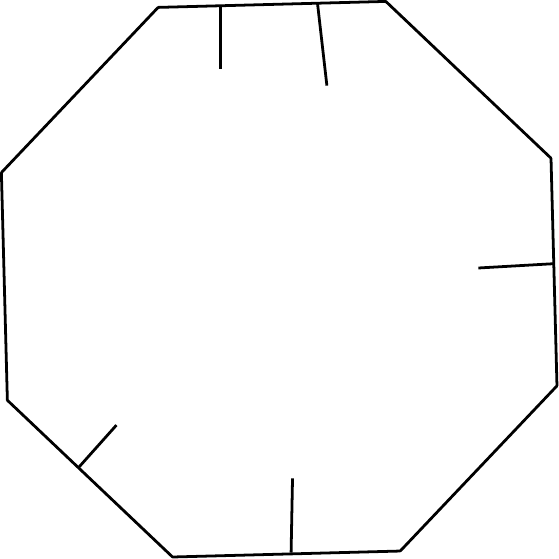 
\caption{This octagon has value $\{\{2|0\}, \{2||1*| -1*\} |
 \{0| -3\}, \{1*| -1*|| -3\}\}$
}
\label{octagon-t-s}
\end{figure}

\begin{figure}
\centering
\def \svgwidth{4in}
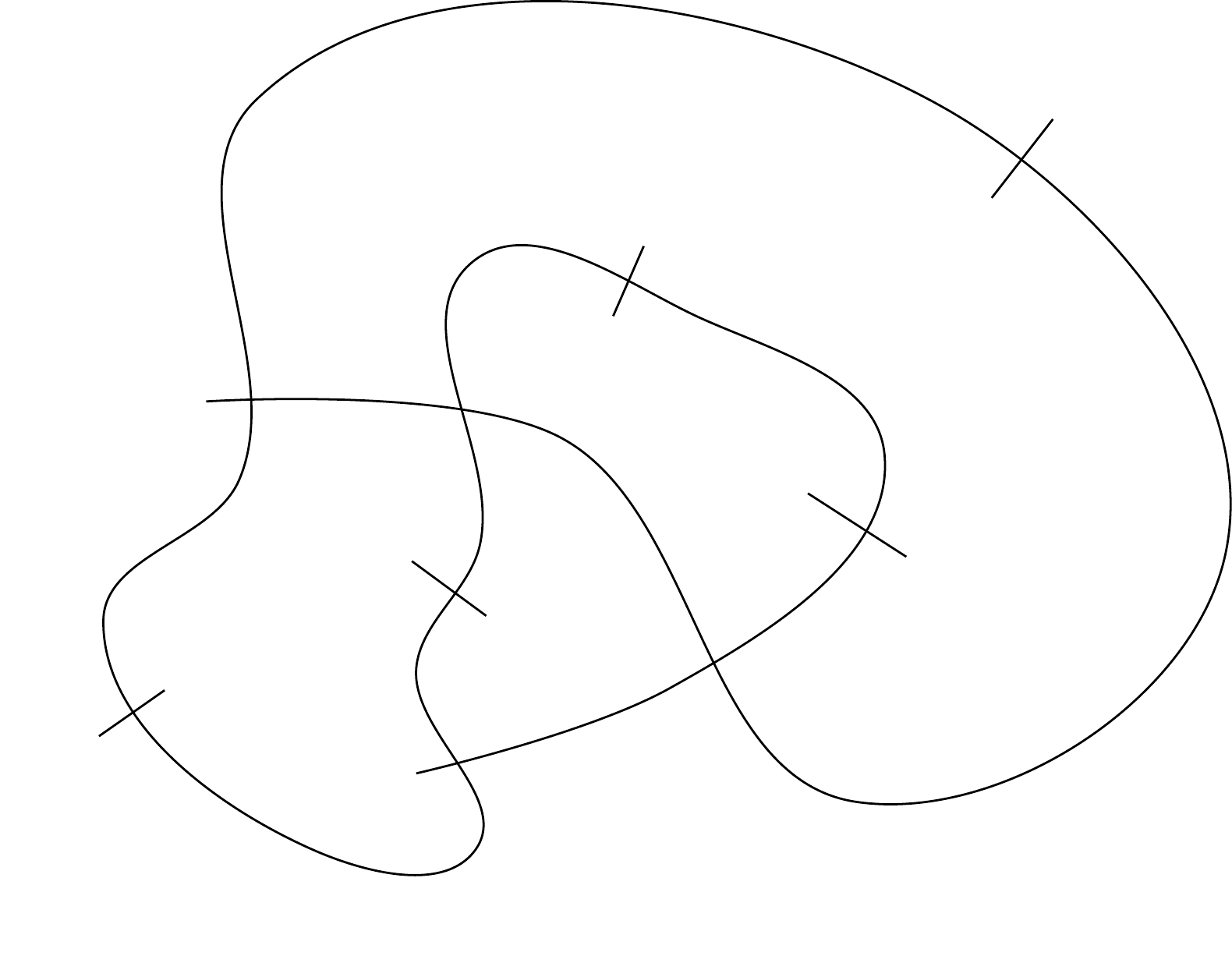 
\caption{This position can be decomposed as the sum of its five regions (faces).  Its value is the sum of the values of the five regions.}
\label{sample-calculation-t-s}
\end{figure}

\subsection{Well-tempered Scored Clobber}
\textsc{Clobber}, first presented in \cite{Clobber},
is an all-small partizan game played between Black and White using black and white checkers on a square grid.  On your turn,
you can move one of your pieces onto an orthogonally adjacent piece of the enemy color, which is removed from the game.
Eventually, no pieces of opposite colors are adjacent, at which point neither player can move and the game ends.
The last player able to move is the winner.  A common starting position is a board full of alternating black and white checkers.

One way to make this into a scoring game would be to assign scores for captured pieces.  This fails to make a \emph{well-tempered} scoring game,
so we do something different.  Say that a checker of color $C$ is \emph{isolated} if it belongs to a connected group of checkers of color $C$, none
of which is adjacent to a checker of the opposite color.  In other words, a checker is isolated if it has no chance of being taken by an enemy checker.

We then add a new \emph{collecting} move, in which you can collect an isolated checker of either color, removing it from the board.  With this rule, the game ends
when no checkers remain on the board.  Every move (clobbering or collecting) decreases the total number of checkers by exactly one.
If the initial position has $n$ checkers,
the game will last exactly $n$ moves.  So this is a well-tempered game.

In \textsc{Mercenary Clobber}, you get one point for collecting an enemy checker, and zero points for collecting one of your own.  Some example
positions and their values are shown in Figure~\ref{mercenary-clobber-dictionary}

\begin{figure}
\centering
\def \svgwidth{4in}
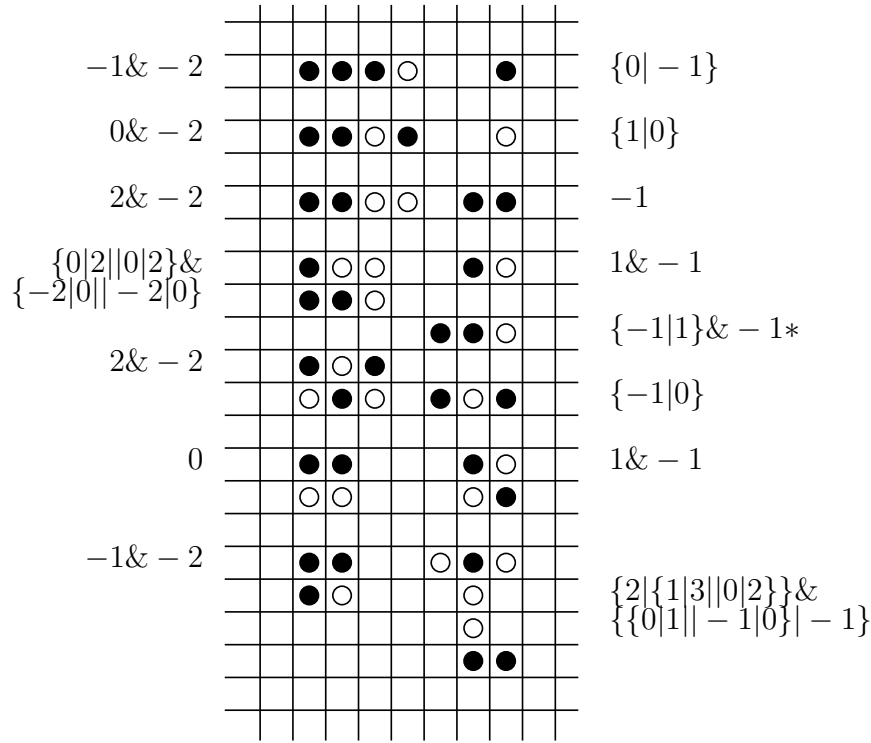 
\caption{Sample positions of \textsc{Mercenary Clobber}, in which you get one point for collecting an enemy piece and zero points for collecting
one of your own pieces.  Scores are from Left = Black's point of view.}
\label{mercenary-clobber-dictionary}
\end{figure}

In \textsc{Celestial Clobber}, the two players are GoLd and EaRth (Left and Right).  Golden pieces are worth 1 point for whoever collects them,
and earthen pieces are worth zero points.  Some example positions and their values are shown in Figure~\ref{penny-farthing-dictionary}.

\begin{figure}
\centering
\def \svgwidth{4in}
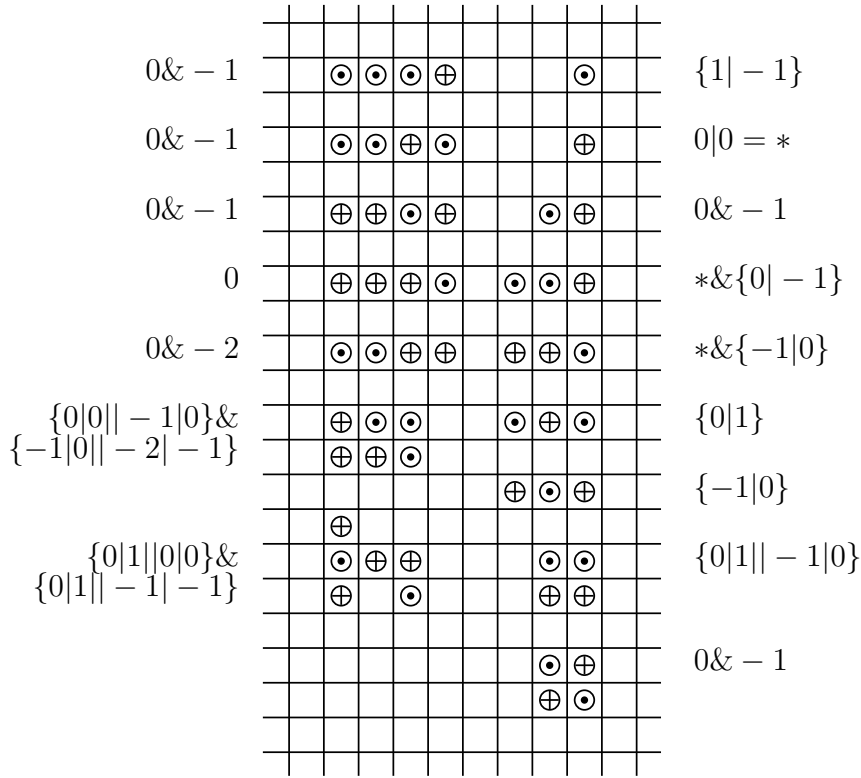 
\caption{Sample positions of \textsc{Celestial Clobber}, in which golden pieces are worth 1 point and earthen pieces are worth none.
Here, goLden pieces are denoted $\odot$ while eaRthen pieces are denoted $\oplus$.  Scores are from Left = Gold's point of view}
\label{penny-farthing-dictionary}
\end{figure}

\subsection{To Knot or Not to Knot}
The game \textsc{To Knot or Not to Knot}, introduced in~\cite{KnotGames}, is a game played using knot pseudodiagrams.  In knot theory, it
is common to represent three-dimensional knots with two-dimensional \emph{knot diagrams}, such as those in Figure~\ref{fig:knot-diagrams}.
Each crossing graphically expresses which strand is on top.  In a \emph{knot pseudodiagram}, ambiguous crossings are allowed, in which it is unclear
which strand is on top.  Figure~\ref{fig:pseudodiagrams} shows two examples.

\begin{figure}
\centering
\subfloat[knot diagrams]{\label{fig:knot-diagrams}
\def \svgwidth{2in}
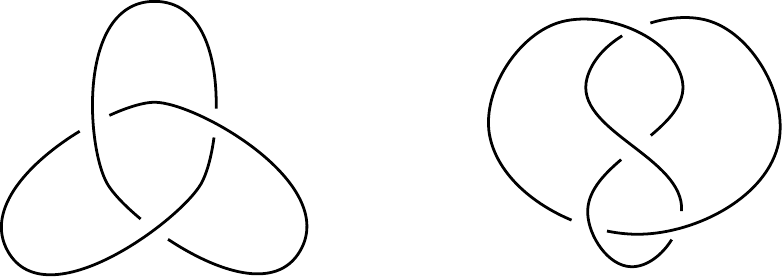}
\subfloat[pseudodiagrams]{\label{fig:pseudodiagrams}
\def \svgwidth{2in}
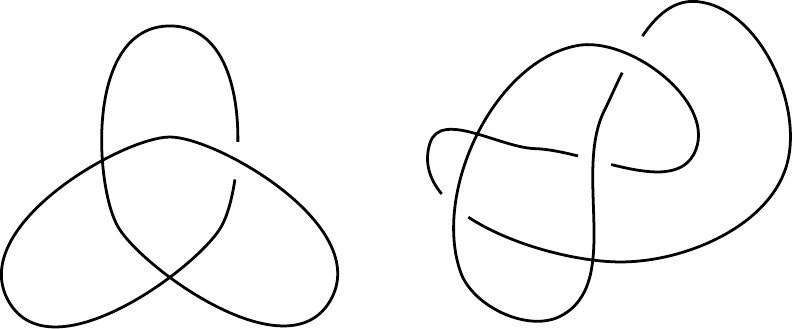}
\caption{Knot diagrams and pseudodiagrams.  In \ref{fig:pseudodiagrams}, the intersections marked as $*$ are ambiguous, unresolved crossings,
where we do not specify which strand is on top.}
\label{knot-diagrams-and-pseudodiagrams}
\end{figure}

Knot Pseudodiagrams were introduced by Hanaki~\cite{Hanaki}.
The motivation for studying pseudodiagrams comes from fuzzy electron microscopy images of DNA strands.
The motivation for the game of \textsc{To Knot or Not to Knot} is less clear:

Two players, \emph{King Lear} and \emph{Ursula} (also known as \emph{Knotter} and \emph{Unknotter}, or Left and Right)
start with a knot pseudodiagram in which every crossing is unresolved.  Such a pseudodiagram is called a \emph{knot shadow}.
They take turns alternately resolving an unresolved crossing.  (This game is best played on a chalkboard.)  Eventually, every crossing is resolved,
yielding a genuine knot diagram.  The Unknotter wins if the knot is topologically equivalent to the unknot, and the Knotter wins otherwise.
A sample game is shown in Figure~\ref{tkontk-game}.

\begin{figure}
\centering
\def \svgwidth{4in}
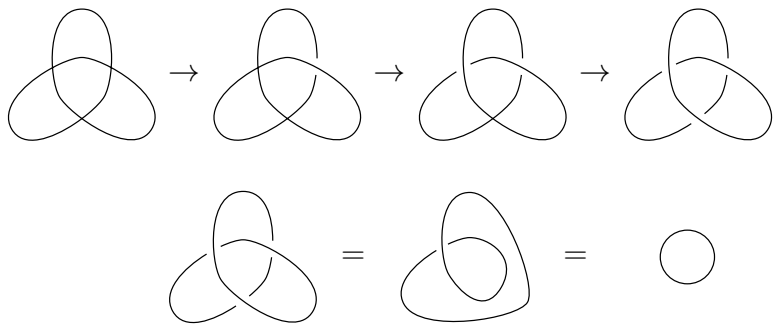
\caption{A sample game of \textsc{To Knot or Not to Knot}.  The top row shows the actual sequence of play,
which takes three moves.  The bottom row shows why the Unknotter has won.}
\label{tkontk-game}
\end{figure}

The connection with combinatorial game theory comes from positions such as the one shown in
Figure~\ref{tkontk-decomposition}, which decompose as connected sum of smaller pseudodiagrams.
From knot theory, we know that a connected sum of knots will be unknotted if and only if every summand is unknotted.\footnote{See pages 99-104 of
\cite{KnotBook} for an outline of a proof of this fact.}  Thus the Knotter wins a connected
sum of positions by winning at least one of the summands.

\begin{figure}
\centering
\def \svgwidth{4in}
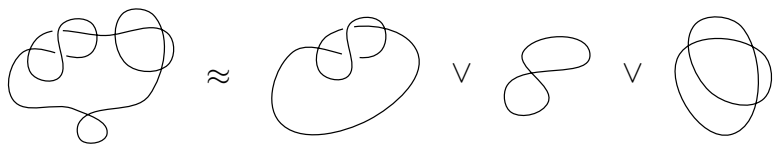
\caption{The position on the left decomposes as a connected sum of the three positions on the right.}
\label{tkontk-decomposition}
\end{figure}

This shows that a connected sum of positions is strategically the same as the disjunctive \textsc{or} (the operation $\vee$ of \S \ref{sec:boolean}).
Since the Knotter is Left and the
Unknotter is Right, we assign a value of $1$ to a Knotter victory and a value of $0$ to an Unknotter victory.  With these conventions,
Figure~\ref{tkontk-dictionary} shows the values of a few small positions.

\begin{figure}
\centering
\def \svgwidth{4in}
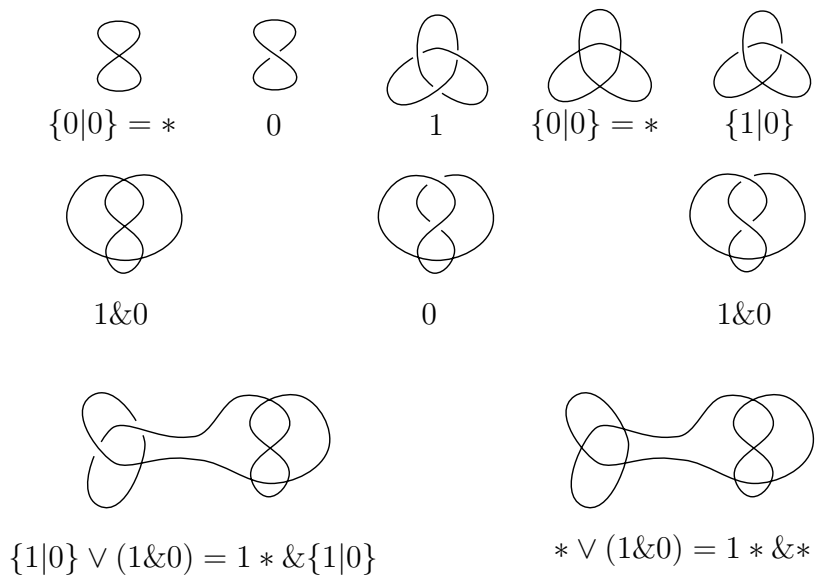
\caption{The values of a few positions of \textsc{To Knot or Not to Knot}, from the Knotter (= Left)'s point of view.}
\label{tkontk-dictionary}
\end{figure}

One difficulty with \textsc{To Knot or Not to Knot} is the problem of determining which player has won at the game's end!  Determining
whether two knots are equivalent is no easy problem.
However, there is a certain class of knot diagrams which has an efficient test for the unknot.  These are the \emph{rational knots}.
A \emph{rational tangle} is a tangle (knot with loose strands) built up from the initial tangles of Figure~\ref{fig:base-tangles} by the operations of adding
twists to the bottom or right sides, as in Figure~\ref{fig:tangle-construction}.
\footnote{It turns out that adding twists on the left or top is equivalent to adding twists
on the right or bottom.}  A rational knot is one
obtained from a rational tangle by closing up the top strands and the bottom strands, as in Figure~\ref{fig:tangle-closure}.

\begin{figure}
\centering
\subfloat[The base tangles]{\label{fig:base-tangles}
\def \svgwidth{1in}
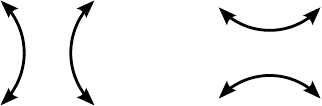}
\subfloat[Construction of new tangles]{\label{fig:tangle-construction}
\def \svgwidth{2in}
\small
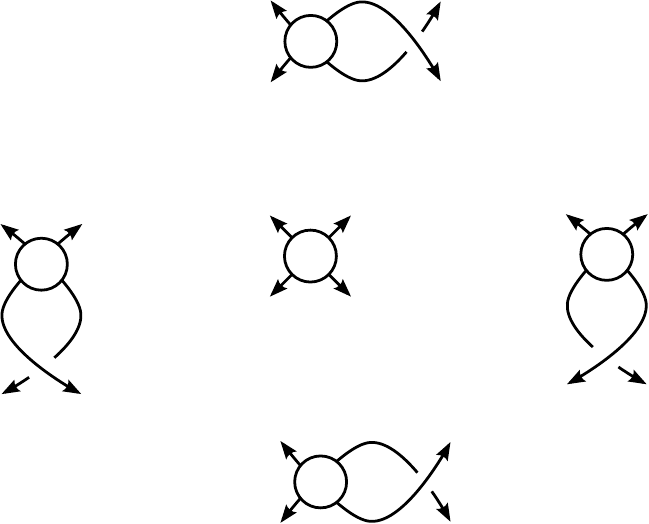}
\subfloat[Turning a tangle into a knot]{\label{fig:tangle-closure}
\def \svgwidth{1in}
\small
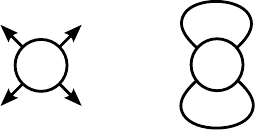}
\caption{The construction of rational tangles.  The rational tangles are the smallest class of tangles containing
the two tangles of \ref{fig:base-tangles} and closed under the four operations of \ref{fig:tangle-construction}.  A rational knot
is one obtained from a rational tangle by closing off a tangle, as in \ref{fig:tangle-closure}.  See \ref{tangle-example} for an example.}
\label{rational-tangles}
\end{figure}

\begin{figure}
\centering
\subfloat{\label{fig:rational-tangle-example-a}
\def \svgwidth{4in}
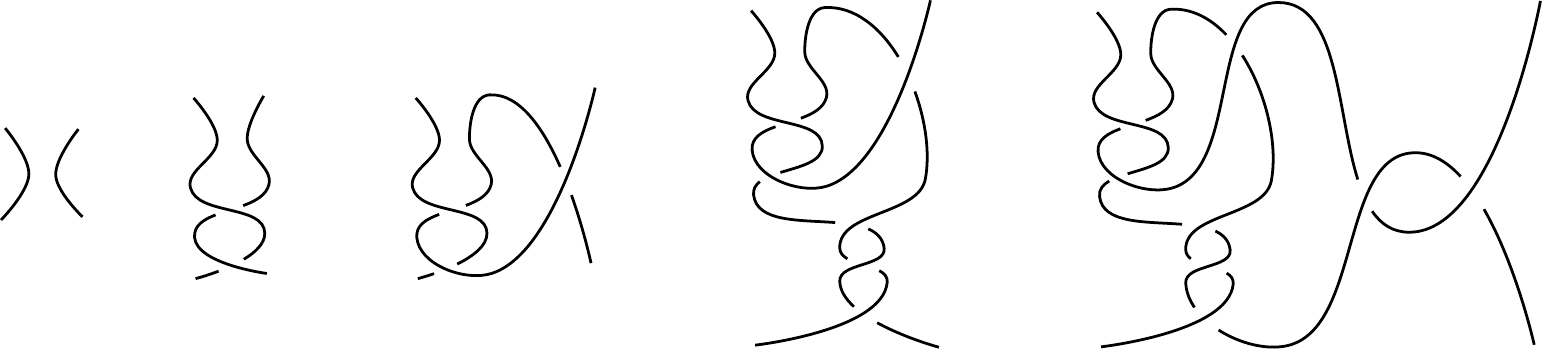} \\
\subfloat{\label{fig:rational-tangle-example-b}
\def \svgwidth{2in}
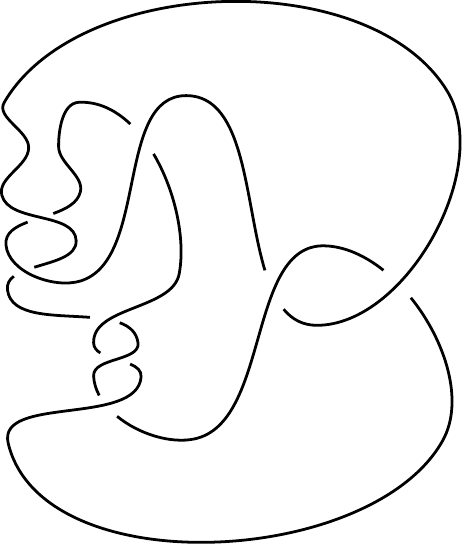}
\caption{An example construction of a rational knot.}
\label{tangle-example}
\end{figure}

Figure~\ref{fig:rational-tangle-example-a} shows the construction of a typical rational tangle, and Figure~\ref{fig:rational-tangle-example-b} shows
the resulting knot.
There is a simple way to classify rational knots using continued fractions; see \cite{KnotBook} or \cite{Conway1970} for details.  Consequently,
\emph{rational knot shadows} like those shown in Figure~\ref{rational-shadows}
become playable games\footnote{In the sense that at the game's end, it is algorithmically possible to identify the winner.}:

\begin{figure}
\centering
\def \svgwidth{4in}
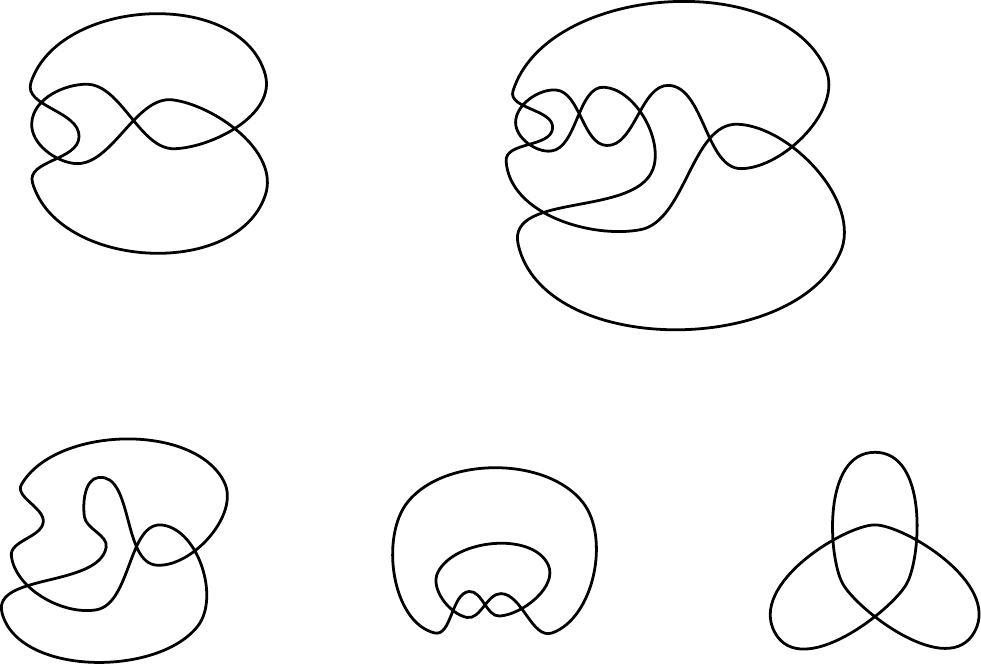
\caption{Rational shadows, on which we can play \emph{TKONTK}.}
\label{rational-shadows}
\end{figure}

It turns out that there is a general rule for determining the value of a rational shadow.\footnote{Let $\left[n_1,n_2,\ldots,n_m\right]$ denote the rational knot shadow
obtained by adding $n_1$ vertical twists, $n_2$ horizontal twists, $n_3$ vertical twists, and so on.  Then we have some reductions for simplifying
the game:
\[ \left[0,n+1,n_2,\ldots \right] = * + \left[0,n,n_2,\ldots\right] \]
\[ \left[\ldots,n_{m-2},n+1,0\right] = * + \left[\ldots,n_{m-2},n,0\right]\]
\[ \left[0,0,n_2,\ldots \right] = \left[ n_2, \ldots \right] \]
\[ \left[\ldots, n_{m-2},0, 0\right] = \left[ \ldots, n_{m-2} \right] \]
\[ \left[\ldots, n_{k-1}, 0, n_{k+1}, \ldots\right] = \left[\ldots, n_{k-1} + n_{k+1}, \ldots \right]\]
\[ \left[1,n_1,\ldots\right] = \left[1+n_1,\ldots\right]\]
\[ \left[\ldots,n_{m-1},1\right] = \left[\ldots,n_{m-1} + 1\right]\]
After applying these rules as much as possible, one reaches a minimal form $\left[n_1,n_2,\ldots,n_m\right]$.  Then if
all of the $n_k$ for $1 < k < m$ are even, and $n_1 + n_m$ is odd, then this position has value $*$.  Otherwise, it has value $1\&0$
or $(1+*)\&*$ depending on whether $n_1 + n_2 + \cdots + n_m$ is even or odd.  Note that this rule determines the value of rational \emph{shadows},
and does not apply to arbitrary rational \emph{pseudodiagrams}.
See \cite{Me} for further information.}

The only $u$-values that I have seen in \textsc{To Knot or Not to Knot} are $0, 1$ and $\frac{1}{2}*$.
As noted above, \textsc{To Knot or Not to Knot} the motivation for considering well-tempered scoring games.
The fact that $5$ of the $8$ possible $u$-values are missing is therefore disappointing.

\textsc{To Knot or Not to Knot} is an example of a Boolean-valued well-tempered scoring game that can be played with any boolean function
$f : \{0,1\}^n \to \{0,1\}$.  Thinking of $f$ as a black box with $n$ inputs and one output, we let the two players Left and Right take turns
setting the values of the inputs to $f$.  After $n$ moves, the inputs are all determined, and then the output of $f$ specifies the winner.
There is no stipulation that Left can only set inputs to 1 or that Right can only set inputs to 0, or that the players work through the inputs
in any particular order.  We call this the \textsc{Input-setting Game} for $f$.
\textsc{To Knot or Not to Knot} is an instance of this game where $f$ is the function which takes the crossing information and outputs
a bit indicating whether the final knot is knotted.

The scarcity of values occurring in \textsc{To Knot or Not to Knot} led me to initially conjecture that all input-setting games can only
involve the $u$-values $0, 1$, and $\frac{1}{2}*$.  However, it turns out that every Boolean-valued game is equivalent to an input-setting game
for some boolean function $f$.\footnote{
Consider Bill Taylor and Dan Hoey's game of \textsc{Projective Hex}~\cite{ProjectiveHex}, played on the faces of a dodecahedron or
the vertices of an icosahedron.  The same strategy-stealing argument for \textsc{Hex} shows that dodecahedral \textsc{Projective Hex} is a first-player
win.  By symmetry, there are no bad opening moves.  We can view this game as an input-setting game for some six-input binary function $f(x_1,\ldots,x_6)$.
Let $g(x_1,\ldots,x_7) = f(x_1,\ldots,x_6) \oplus x_7$, where $\oplus$ denotes exclusive-or.
Then the input-setting game of $g$ is a second-player win.  (This is not difficult
to show, given what we just said about dodecahedral \textsc{Projective Hex}.)  By examining Theorem~\ref{boolean-summary}, one can verify that
any odd-tempered Boolean-valued game that is a second-player win must have $u^+ = u^- = 1/2$.  This establishes the existence of input-setting
games with $1/2$ as a $u$-value.  Combining this game with the positions appearing in Figure~\ref{tkontk-dictionary} via $\vee$ and $\wedge$,
one can in fact get all 35 pairs of possible $u$-values, without leaving the realm of boolean input-setting games.}
This leads me to suspect that the same holds
within \textsc{To Knot or Not to Knot}, since there are no obvious restrictions on which boolean functions can be realized by a pseudodiagram.

\subsection{Hex, Misere Hex, Rebel Hex}
\textsc{Hex} is a well-known game due to Piet Hein and John Nash.\footnote{A good source on Hex is Browne's book~\cite{HexBook}.}
Two players, Black and White, alternately
place pieces on the vertices of the board shown in Figure~\ref{fig:hex-board}.
Black wins by forming an unbroken chain of black pieces connecting the left and right sides (as in Figure~\ref{fig:hex-win}),
while White wins by forming a chain
of white pieces connecting the top and bottom sides.  It turns out that the board cannot fill up without one player winning.\footnote{This fact is
related to the Brouwer fixed-point theorem.  See Gale's article~\cite{GaleBrouwerHex}, for example.}  Therefore, ties are impossible.

\begin{figure}
\centering
\subfloat[an $11 \times 11$ Hex board]{\label{fig:hex-board}
\def \svgwidth{2in}
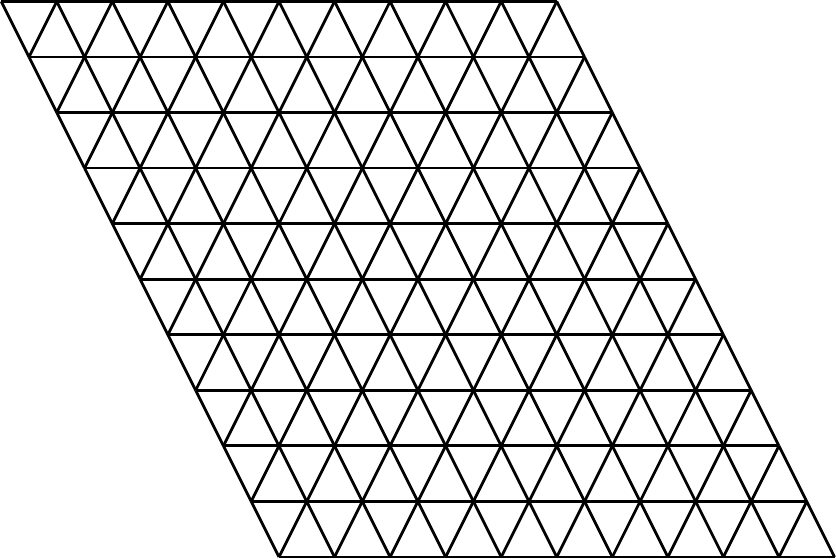}
\subfloat[a winning chain for Black]{\label{fig:hex-win}
\def \svgwidth{2in}
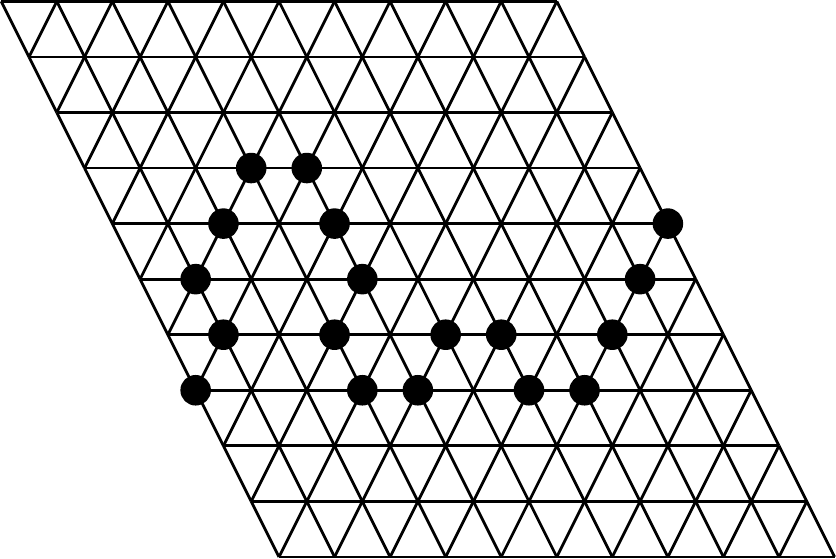}
\caption{An $11 \times 11$ board of Hex, and an example of a winning chain for Black.  We will always assume that Black is trying to connect
the right and left sides, while White is trying to connect the top and bottom sides.  In practice,
other board sizes are often used, such as $10 \times 10$ and $14 \times 14$.}
\label{hex-rules}
\end{figure}

In Figure~\ref{hex-connect}(a), Black can guarantee a connection between the two black pieces.  If White moves at one of the intervening
spaces, then Black can move at the other and complete the connection.  The configuration of Figure~\ref{hex-connect}(a) is called a \emph{bridge}.
Chains of bridges can provide guaranteed connections across long distances.  For example, in Figure~\ref{hex-connect}(b), Black can guarantee victory.

On the other hand, the connection of Figure~\ref{hex-connect}(c) is a weaker connection.  With one more move, Black can establish a bridge, and guarantee
a connection between her pieces.  However, White can also connect his pieces via a bridge with one move.
Thus whichever player moves next can control the connection.  However, when two of these configurations are place in parallel, as in
Figure~\ref{hex-connect}(d),
the result is a safe connection for Black, because if White interferes with the top connection, Black can respond by solidifying the bottom connection,
and vice versa.

\begin{figure}
\centering
\def \svgwidth{4.5in}
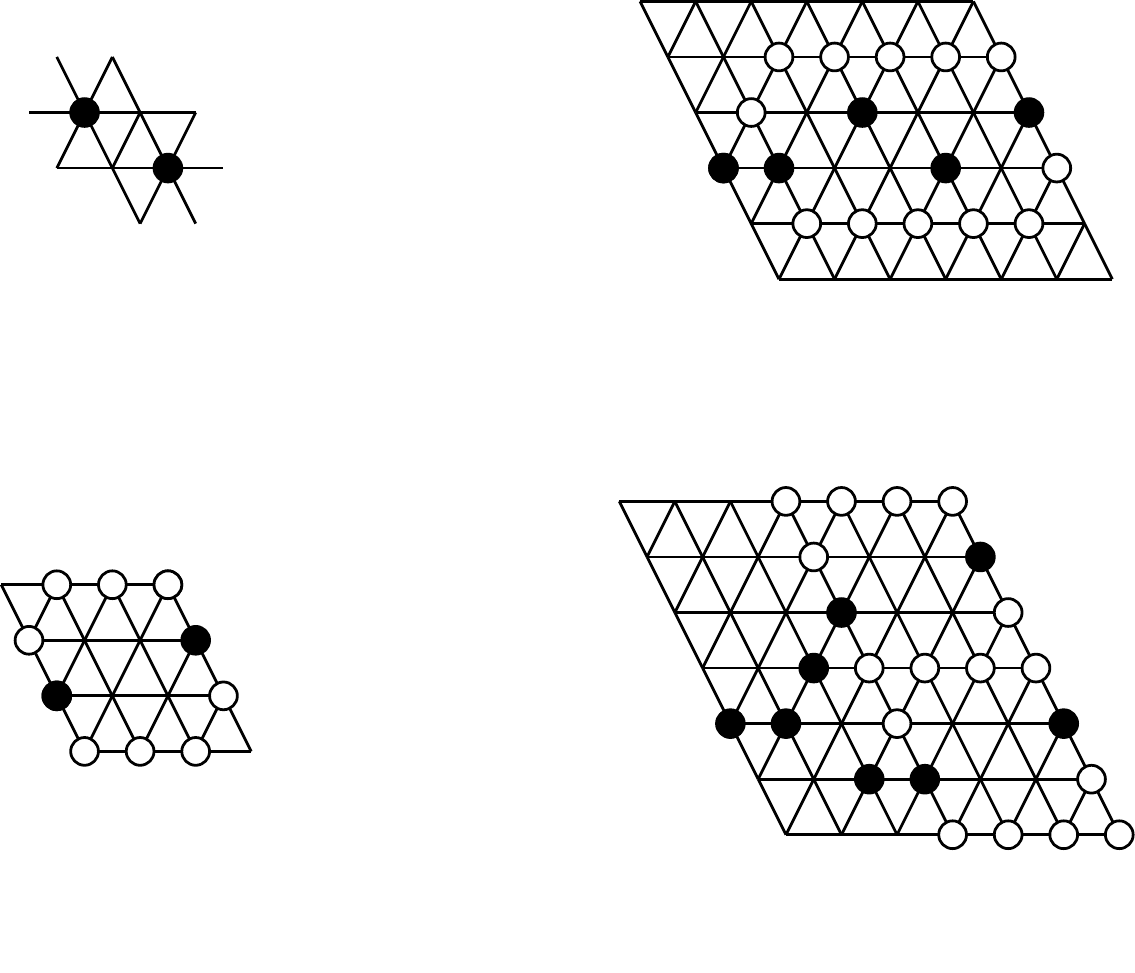
\caption{Examples of connections of various strengths.}
\label{hex-connect}
\end{figure}

These ways of chaining positions together in series or parallel correspond to the operations $\wedge$ and $\vee$ of \S\ref{sec:boolean},
if we assign a value of 1 to a successful Black connection and a value of 0 to a successful White connection.
Figure~\ref{hex-dictionary} shows a number of small positions and their values, with these conventions.

\begin{figure}
\centering
\def \svgwidth{4.5in}
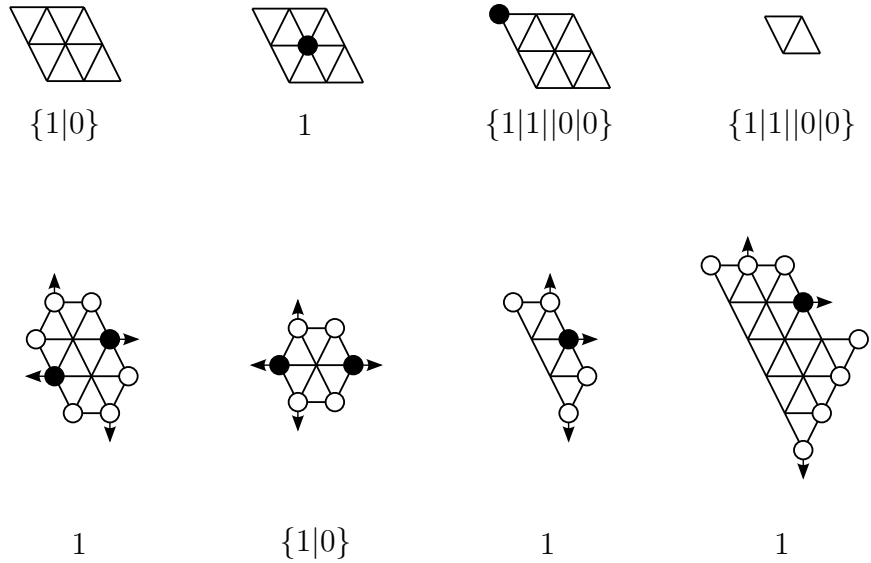
\caption{Sample positions in \textsc{Hex}, from the point of view of Black = Left = Horizontal.  Black victory is scored as 1, while White victory is scored as 0.
An arrow on a piece indicates that the piece is connected to the edge in the specified direction.}
\label{hex-dictionary}
\end{figure}

The reader will notice that the only $u$-values which occur are $0$, $1$, and $\frac{1}{2}*$, and that every value is in $\mathcal{I}$.
In other words, the only even-tempered
values are $0, 1,$ and $\{1*||*\}$ and the only odd-tempered values are $*, 1*,$ and $\{1|0\}$, and every game is already invertible.
This follows from the fact that in Hex, \emph{moving never hurts you}.  Consequently,
if $G$ is a position, then $\lout(G^L) \ge \lout(G) \ge \lout(G^R)$ and $\rout(G^L) \ge \rout(G) \ge \rout(G^R)$ for any left option
$G^L$ and right option $G^R$.  This prevents the values $\{0|1\}$ and $\{*|1*\}$ from ever occurring, and these in turn are necessary to produce
any of the other missing values.

A variant of Hex is \textsc{Misere Hex}, which is played under identical rules, except that the player who connects his two sides \emph{loses}.
As in Hex, positions can be chained together in series or parallel, and Section \ref{sec:boolean} becomes applicable.\footnote{It is unclear that
these configurations occur in actual games of Misere Hex, however.}  Assigning a value of 1 to a position with a white connection, and 0 for a position
with a black connection, some sample positions and their values are shown in Figure~\ref{misere-hex}.

\begin{figure}
\centering
\def \svgwidth{4.5in}
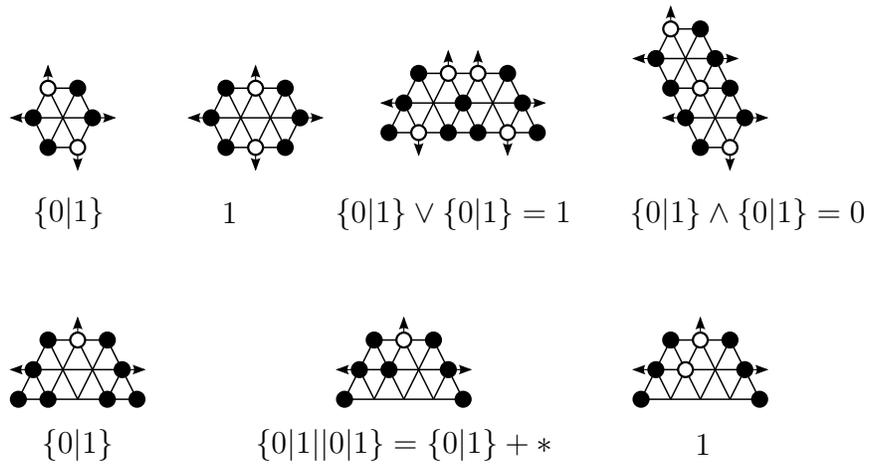
\caption{Sample positions in \textsc{Misere Hex}, from the point of view of Black = Left = Horizontal.  Black victory is scored as 1, while White victory is scored as 0.
An arrow on a piece indicates that the piece is connected to the edge in the specified direction.}
\label{misere-hex}
\end{figure}

Again, we notice a scarcity of values.  Lagarias and Sleator~\cite{Lagarias} show that $n \times n$ Misere Hex is a win for the first player when $n$ is even,
and for the second player when $n$ is odd.  Using their techniques, it is easy to show that more generally:
\begin{itemize}
\item If $G$ is an even-tempered position, then $G$ cannot be a second-player win.
\item If $G$ is an odd-tempered position, then $G$ cannot be a first-player win.
\item If $G$ is any position, then $G^L \gtrsim G + * \gtrsim G^R$ for every $G^L$ and $G^R$.
\item If $G$ is any position, then $u^+(G) = u^-(G) \in \{0,1,\frac{1}{2}\}$.  In other words, the only possible values
are
\[ 0, 1, \{0|0\}, \{1|1\}, \{0|1\}, \{0|1||0|1\}.\]
In fact, no new games can be built from these without contradicting one of the preceding bullet points.
\end{itemize}

To get more interesting values, we consider a third variant of Hex, \textsc{Rebel Hex}.
This is played the same as Hex, except for the following additional rule: whenever a piece is played directly between two enemy pieces, it changes colors.
For example in Figure~\ref{rebel-hex-rules}(a), if
White plays a piece at the $*$, it will immediately change colors to black.
On the other hand, if Black plays at $*$, her piece will remain black.
In Figure~\ref{rebel-hex-rules}(b), if either player moves at $*$, the piece will change colors.
For a color reversal to occur, the surrounding enemy pieces must be directly adjacent and on opposite
sides. Thus in Figure~\ref{rebel-hex-rules}(c), no reversal will occur if either player moves at $*$.
Moreover, playing a piece can never change the color of a piece already on the board,
so in Figure~\ref{rebel-hex-rules}(d), if Black moves at $*$, the white piece remains white, although it is now surrounded.

\begin{figure}
\centering
\def \svgwidth{3in}
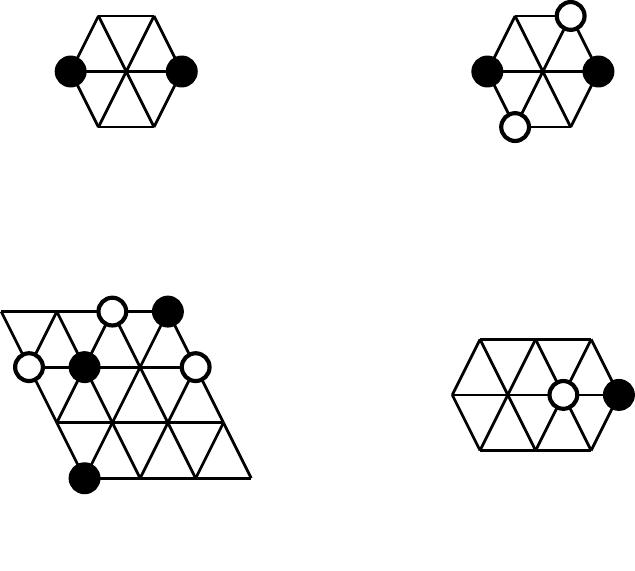
\caption{Rules of \textsc{Rebel Hex}.}
\label{rebel-hex-rules}
\end{figure}

With this new rule in place, new values can occur, as shown in Figure~\ref{rebel-hex-dictionary}.
For example, we can now have quarter-connections and $3/8$-connections.  If the players play a thousand copies
of Figure~\ref{fig:quarter-connection} in parallel and Black tries to win as many as possible, she will be able to win about 250 of them.
The significance of these partial connections is exhibited when we chain them together in series or parallel.
For example, White's $\frac{3}{4}$-strength connection in Figure~\ref{fig:quarter-connection} ensures that White will win, regardless of who moves next.
But when four copies of this configuration are placed in series, as in Figure~\ref{fig:four-quarters}, the resulting connection has strength
\[ \frac{3}{4} \cap \frac{3}{4} \cap \frac{3}{4} \cap \frac{3}{4} = 0,\]
and now Black is guaranteed a victory!
\begin{figure}
\centering
\def \svgwidth{4.5in}
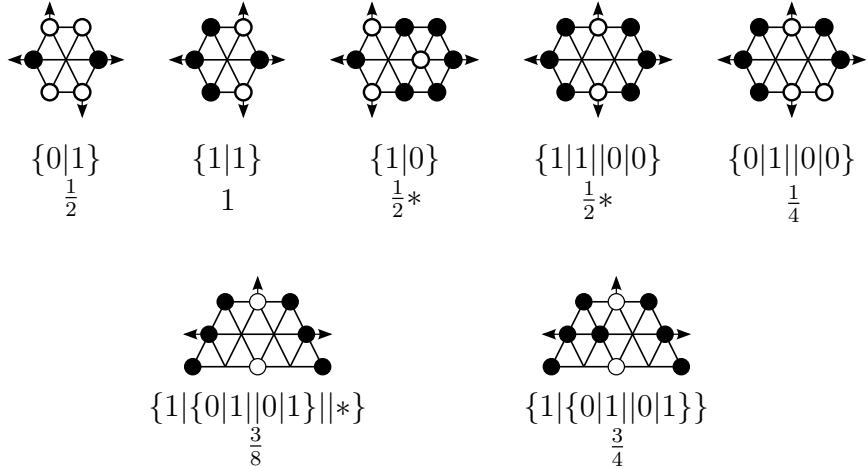
\caption{Sample positions in \textsc{Rebel Hex}, from the point of view of Black = Left = Horizontal.  Black victory is scored as 1, while White victory is scored as 0.
An arrow on a piece indicates that the piece is connected to the edge in the specified direction.  The positions on the bottom row
are assumed to be along the bottom edge of the board.  Each position is given with its value and its cooled $u$-value.}
\label{rebel-hex-dictionary}
\end{figure}

\begin{figure}
\centering
\subfloat[~]{\label{fig:quarter-connection}
\def \svgwidth{1in}
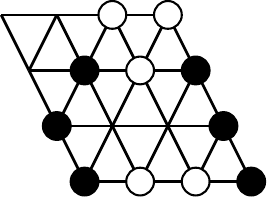}
\subfloat[~]{\label{fig:four-quarters}
\def \svgwidth{3in}
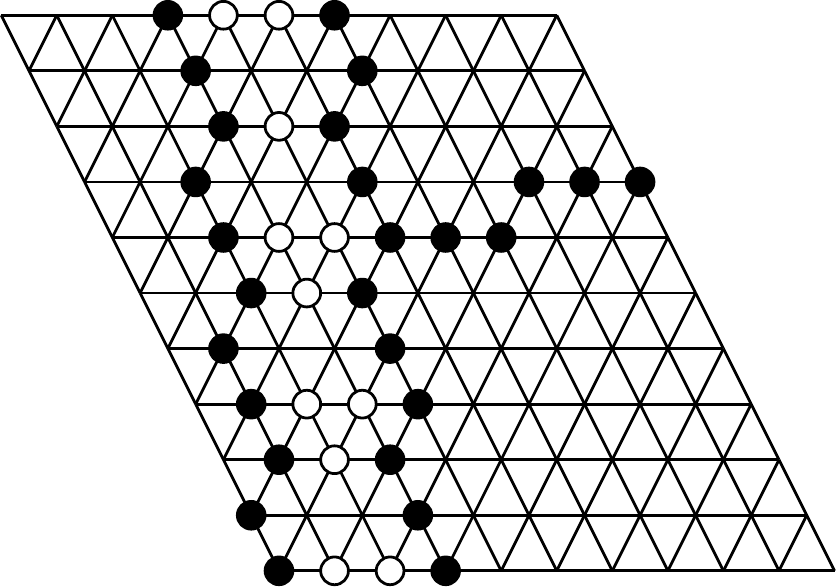}
\caption{The value of \ref{fig:quarter-connection} is $* + \{0|1||0|0\} = \{\{0|1||0|1\}|0\}$ which is a win for White no matter which player goes first.
Note that this is a $u$-value of $\frac{3}{4}$ for White, or $\frac{1}{4}$ for Black.
On the other hand, \ref{fig:four-quarters} has $u$-value $1$ for Black, or $0$ for white, because $\frac{1}{4} + \frac{1}{4} + \frac{1}{4} + \frac{1}{4} = 1$.
Therefore \ref{fig:four-quarters} is a guaranteed win for Black. }
\label{quarters}
\end{figure}

\section{Conclusion}
Sadly, none of our theory seems to generalize to scoring games which are not well-tempered.  Central notions such as $\approx_\pm$,
upsides, and downsides only make sense in a setting where the identity of the final player is predetermined, which generally does not occur.
If $G$ and $H$ are
well-tempered scoring games such that $G \approx H$, it may well be the case that $G + X$ and $H + X$ have different outcomes for some (non-well-tempered)
scoring game $X$.\footnote{One case where this does not occur is the following: if $G$ is well-tempered and
$G \gtrsim 0$, then $\lout(G + X) \ge \lout(X)$ and $\rout(G + X) \ge \rout(X)$ for all scoring games $X$, well-tempered or not.
This follows because the \emph{well-tempered} game $X = \{-N|N\}$ for $N \gg 0$ provides a sufficient test game for comparing an arbitrary game
to $0$.}

One approach to general scoring games might be to try heating them until
$\lout(H) \ge \rout(H)$ is satisfied for all subgames $H$.
There are a couple of problems with this approach.  First of all, a game like $\{*|1*\}$ will not satisfy this property
under any amount of heating, because $\{2t|0||1|-2t\}$ always has left outcome $0$ and right outcome $1$.  A bigger problem is
that after sufficient cooling,
the theory of scoring games is equivalent to the theory of misere all-small games.\footnote{If $G$ is a scoring game and $N \gg 0$, then
the left outcome of $\int^{-N} G$ is approximately $-N$ or approximately $0$, according to whether Left loses or wins $G$ played as a misere game
with Left going first.}
Since the theory of misere all-small games is somewhat hopeless, cooling
had better not be a well-defined operation on scoring games!  But if it is not, heating becomes a lossy operation, and its usefulness diminishes.
\subsection{Acknowledgments}
This work was done in part during 
REU programs at the University of Washington in 2010 and 2011, and in part while supported by the NSF Graduate Research Fellowship Program in the Autumn of 2011.
The author would like to thank Rebecca Keeping, James Morrow, Neil McKay, and Richard Nowakowski, with whom he discussed the content of this paper, as well as
Allison Henrich, who introduced the author to \textsc{To Knot or Not to Knot}, which ultimately led to this investigation.

\bibliographystyle{plain}
\bibliography{Master}{}
\end{document}